\documentclass[11pt,reqno]{amsart}
\usepackage[usenames, dvipsnames]{color}
\numberwithin{equation}{section}

\newtheorem{theorem}{Theorem}[section]
\newtheorem{definition}[theorem]{Definition}
\newtheorem{proposition}[theorem]{Proposition}
\newtheorem{lemma}[theorem]{Lemma}
\newtheorem{corollary}[theorem]{Corollary}
\newtheorem{remark}[theorem]{Remark}

\def\al{\aligned}
\def\eal{\endaligned}
\def\be{\begin{equation}}
\def\ee{\end{equation}}
\def\lab{\label}
\def\a{\alpha}

\def\e{\epsilon}
\def\td{\tilde}

\def\M{{\bf M}}

\def\al{\aligned}

\def\g{\bar}
\def\pa{\partial}
\def\nb{\nabla}

\def\lam{\lambda}

\def\mb{\mathbb}

\DeclareMathOperator{\Ric}{Ric}

\DeclareMathOperator{\Hess}{Hess}

\numberwithin{equation}{section}

\usepackage{cite}
\usepackage[backref=page]{hyperref}

\begin{document}

\title[K\"ahler Ricci flow]{Laplace comparison on K\"ahler Ricci flow and convergence}
\author[TZZZZ]{Gang Tian, Qi S. Zhang, Zhenlei Zhang, Meng Zhu and Xiaohua Zhu}
\address{BICMR and School of Mathematics, Peking University, Beijing, 100871, China}
\email{gtian@math.pku.edu.cn}
\address{Department of Mathematics, University of California, Riverside, CA 92521, USA}
\email{qizhang@math.ucr.edu}
\address{Department of Mathematics, Department of Mathematics, Capital Normal University, Beijing}
\email{zhleigo@aliyun.com}
\address{School of Mathematical Sciences, East China Normal University, Shanghai 200241, China}
 \email{mzhu@math.ecnu.edu.cn}
\address{BICMR and School of Mathematics, Peking University, Beijing, 100871, China}
\email{xhzhu@math.pku.edu.cn}
\date{2025.10.19,  MSC2020: 53E20,
		53E30, 53C23}

\begin{abstract}
 We first prove a uniform integral  Laplace comparison result for the K\"ahler Ricci flow on Fano manifolds which depends only on the initial metric. As an application, using  Cheeger-Colding theory and previous results  by some of the authors, we give a direct and independent proof of the Hamilton-Tian conjecture on convergence of K\"ahler-Ricci flows, modulo a codimension 4 singular set. We also expounded on some existing literature on this conjecture.
\end{abstract}
\maketitle
\tableofcontents

\section{Introduction}

In this paper we study  the (normalized) K\"ahler Ricci flows (KRF)
\be
\lab{krf}
\partial_t g_{i\g j} = -  R_{i\g j} + g_{i\g{j}} = \pa_i \pa_{\g j} u, \quad t>0,
\ee on a compact, K\"ahler manifold $M$ of complex dimension
$m=n/2$, with positive first Chern class. We always assume that
the initial metric is in the canonical K\"ahler class $2\pi
c_1(M)$.

Given initial K\"ahler metric $g_{i\g j}(0)$, it is  proved in
\cite{Ca:1} that (\ref{krf}) has a solution for all time
$t$. Recently, many  results concerning long time and uniform
behavior of (\ref{krf}) have appeared. For example, when the
curvature operator or the bisectional curvature is nonnegative, it
is known that solutions to (\ref{krf}) stays smooth when time goes
to infinity (see \cite{CCZ:1} for
examples). In the general case, Perelman (cf \cite{ST:1}) proved
that the scalar curvature $R$ is uniformly bounded, and the Ricci
potential $u(\cdot, t)$ is uniformly bounded in $C^1$ norm, with
respect to $g(t)$. When the complex dimension $m=2$, let $(M,
g(t))$ be a solution to (\ref{krf}),  it is proved in
(\cite{CW:1}) that the flow sequentially converges to an orbifold.
When $m=3$, the authors of \cite{TZz:1} proved sequential
convergence except on a small singular set of co-dimension 4. For the general case,
in the paper \cite{TZq1}, it was proved that a Gromov-Hausdorff limit
of the flow is a metric space with volume doubling
 and $L^2$ Poincar\'e inequality and $L^1$ isoperimetric inequality. In \cite{LTZ}, it is proven that the curvature tensor for some KRF can diverge to infinity as time goes to $\infty$.
 The convergence issue
of K\"ahler Ricci flows have attracted much attention recently. See for example
\cite{MS:1}, \cite{PS:1}, \cite{PSSW1:1}, \cite{PSSW2:1},
\cite{Se:1}, \cite{SoT:1},  \cite{CZ:1}, \cite{Sz:1}, \cite{SW:1}, \cite{To:1},
\cite{TZhu1:1}, \cite{TZhu2:1}, \cite{TZZZ:1}, \cite{Zz1:1},
\cite{Zz2:1} and \cite{Zhu:1} and references therein.  Much of these activities are centered around the:

{\it  Hamilton-Tian conjecture (\cite{Ti:1} p35). As $t \to \infty$, the manifolds (M, g(t)) in \eqref{krf} converge (at least in a subsequence) to a shrinking K\"ahler Ricci soliton, except on a set with at least co-dimension 4 singularities.}

As mentioned above, this conjecture was confirmed in complex dimensions 2 and 3 in \cite{CW:1} and \cite{TZz:1} respectively. In the higher dimension cases,
several years ago, two substantial papers \cite{CW:4} and \cite{Ba:1} have appeared where  solutions to the Hamilton-Tian conjecture are presented. See also  Section 4.
  In a more recent paper
\cite{WZ2}, another interesting path to this conjecture is presented, based on two ingredients. The first one is the lower bound of the Bergman functions for \eqref{krf} proven in \cite{Zhkw}, based on \cite{LS}.
The other is the Kodaira embedding of $M$ into $CP^N$ by  sections of holomorphic line bundles. It  appears that  the co-dimension 4 property of the singular set still needs a little extra work except for the K\"ahler Einstein case.

 Most recently in \cite{CMT:1}, certain convergence properties for manifolds whose Ricci curvature is in the Kato class are presented, including volume convergence and metric cone property. For compact manifolds, their condition actually includes the a priori bound that $|Ric|^2$ is in the Kato class in \cite{TZq2}. On p2641, line 5, results in two cited papers [62] and [67] (\cite{Kas} and \cite{KuSh} here) were used to justify the convergence of the Dirichlet energy under Gromov-Hausdorff convergence. However the main results in these two cited papers require either Lipschitz convergence or Ricci curvature being bounded from below by a constant. These conditions are not available in the Kato class setting. So the subsequent results involving convergence of heat kernels and distance, such as Proposition 2.12, seem unavailable yet.
  In any case, the authors wrote that the co-dimension of the singular set of the limit space was not determined c.f. p4 in \cite{CMT:1}. See also Remark \ref{rmk1} below.

 After the work of \cite{TZz:1}, it becomes clear to workers on the topic  that the main obstacle is the lack of integral Laplace comparison result in the spirit of \cite{PeWe1} since the Ricci curvature for each slice of the KRF $(M, g(t))$ lacks a priori $L^p$ bound with $p>n/2$, $n \ge 8$. Although we still can not prove it for $(M, g(t))$, in this paper will  prove the desired integral Laplace comparison result for a conformal metric of $g(t)$. The conformal factor $h(\cdot, t)$ has uniformly bounded $C^{1/2}$ norm and is bounded between two uniform, positive constants. This will allow us to give a direct proof of the conjecture. There are two key ingredients in the proof. One of them is to use the a priori bound that $|Ric|^2$ is in the Kato class (\cite{TZq2}) and the other is to consider a conformal metric $h g(t)$. Here $h$ is determined by an elliptic Schr\"odinger equation to mostly cancel the negative part of the Ricci curvature.
 This is partly inspired  by a similar idea in  \cite{CMT:2}, where the authors treated the "time changed" metric instead of the conformal metric. See Remark \ref{rmk1} below.

 Let us recall the necessary prerequisites first.
Between 2002 and 2015, the following basic properties have been proven by a number of authors. In order to state them, let us fix some notations first. Given two points $x, y \in M$ and a metric $g=g(t)$, we use $d(x, y, t)$ or $d(x, y, g(t))$ or $d_g(x, y)$ to denote the distance, $dg(t)$ or $dg$ for Riemannian  volume element. $B(x, r, g(t))$ or $B(x, r, t)$ denote the geodesic ball centered at $x$ with radius $r$ under $g(t)$, and $|B(x, r, g(t))|_{\td g}$ denote its volume under another metric $\td g$ which may also be $g$ itself. The Ricci curvature is denoted by $Ric_{g(t)}$ and the scalar curvature is denoted by $R(g(t))$. The diameter is denoted by $diam(M, g(t))$. In addition $\Delta_{g(t)}$ is the Laplace-Beltrami operator under $g(t)$ and $\nabla_{g(t)}$ is the covariant derivative. If no confusion arises, we will drop the reference to $g, g(t)$ or $t$ from all these notations.

\noindent {\it Basic Properties KRF.}  Let $(M, g(t))$ be  a K\"ahler Ricci flow (\ref{krf}) on a compact manifold with positive first Chern class.
There exist uniform positive constants $C$ and $\kappa$ depending only on
$g(0)$ so that \\

1.  $| R(g(t))| \le C$. \quad
2. $diam (M, g(t)) \le C.$ \quad
3. $\Vert u(\cdot, t) \Vert_{C^1} \le C.$ Here $u$ is the Ricci potential so that
\be
\lab{ricpot}
R_{i \g j}
= - u_{i \g j} +g_{i \g j}.
\ee

4. $|B(x, r, t)|_{g(t)} \ge \kappa r^n$, for all $t>0$ and $r \in (0, diam (M, g(t)))$.

5. $
|B(x, r, t)|_{g(t)} \le \kappa^{-1}  r^n
$ for all $r>0$, $t>0$.

6. There exists a uniform constant $S_2$ so that the following $L^2$ Sobolev inequality holds: \\
 \[
 \left( \int_{M}  v^{2n/(n-2)} d g(t) \right)^{(n-2)/n} \le S_2 \left( \int_{M} | \nabla v |^2
 dg(t) +  \int_{M} v^2
 dg(t) \right)
\]for all $ v \in C^\infty(M, g(t))$.

7. (a). Let $\Gamma$ be the Green's function on $(M, g(t))$.
Then there exists a uniform constant $C$ such that \be
\lab{GdGjie} |\Gamma(x, y) | \le \frac{C}{d(x, y)^{n-2}}, \quad
|\nabla \Gamma(x, y) | \le \frac{C}{d(x, y)^{n-1}}. \ee
Here $d(x, y)$ is the geodesic distance on $(M, g(t))$.

(b). Let $p=p(x, s, y)$ be the (stationary) heat kernel for
$({M}, g(t))$.  There exist positive constants $a_1$ and $a_2$,
depending only on $g(0)$ such that

\be \frac{a_1}{s^{n/2}} e^{-a_2
d(x, y)^2/s} \le p(x, s, y) \le \frac{1}{a_1 s^{n/2}} e^{- d(x,
y)^2/(a_2 s)}, \qquad s \in (0, 1]; \ee

 \be
 |\nb p(x, s, y)| \le \frac{1}{a_1 s^{(n+1)/2}}
e^{- d(x, y)^2/(a_2 s)}, \qquad s \in (0, 1]. \ee

8. uniform $L^2$ Poincar\'e inequality: for any $v \in
C^\infty(B(x, r))$ where $B(x, r)$ is a proper ball in $({M},
g(t))$, there is a uniform constant $C$ such that \be \lab{l2PI}
\int_{B(x, r)} |v-  v_B |^2 dg \le C r^2 \int_{B(x, r)} |\nb v|^2 dg.
\ee Here and below all quantities are with respect to  $g=g(t)$ and $v_B$
is the average of $v$ in $B(x, r)$.

9.  $K_2$ bound for the Hessian of the Ricci potential $u$: $R_{i \g j}
= - u_{i \g j} +g_{i \g j}$.
\[
\int_M \frac{| Hess \,  u |^2(y)}{d(x, y)^{n-2}} dg(y) \le C.
\]In particular, if one writes $F=|Ric|$, then
\be
\lab{defkato}
K(F^2) \equiv \int_M \frac{| Ric |^2(y)}{d(x, y)^{n-2}} dg(y) \le C.
\ee The notation $K(\cdot)$ stands for Kato type norms frequently used in the study of Schr\"odinger equations.

10. \[ \int_M |Ric(\cdot, t)|^4 dg \le C.\]

Properties  1-3 and 4 are proven by Perelman (c.f. \cite{ST:1}, \cite{P:1}); Property 5 can be found in \cite{Z11:1}
and also \cite{CW:2}; Property 6 was first proven in
\cite{Z07:1} (see also \cite{Ye:1}, \cite{Z10:1} ).
 Properties 7 and 8 are in \cite{TZq1}, proven as a consequence of properties 3, 4, 5, 6, and Property 9 is in \cite{TZq2}. Property 10 is in \cite{TZz:1}.

By going over the proofs of these properties, one knows that the uniform constants in properties 1-9 depend only on the dimension $n$, $C^1$ bound of the initial Ricci potential $u(\cdot, 0)$, $L^\infty$ bound of the initial scalar curvature, initial volume and the constant for the Sobolev inequality for $g(0)$. We refer to these as the basic parameters of the initial metric throughout the paper. A simple but useful observation is the following.

{\remark
\lab{blowupconst}
All the bounds in 1, 3-9 still hold with the same constants if one scale up the metric , i.e. replacing $g$ by $\lam g$ for a constant $\lam \ge 1$. Some of the constants, including that in 9, will become smaller.}

 The first theorem of the paper is

\begin{theorem}\label{thplapc}
Let $(M, g(t))$, $t>0$, be a time slice of the KRF \eqref{krf}. There exists a smooth function $h$ and positive constants $b_1, b_2$ and $b_3$ depending only on the basic parameters of the initial metric such that

 (a). $b_1 \le h \le b_2$;

 (b).  $\Vert h \Vert_{C^{1/2}(M)} \le b_3$. Moreover $\Vert \nb h \Vert_{L^p(M)} \le c_p$ for all $p \ge 1$ with $c_p$ depending only on $p$ and the basic parameters of the initial metric.

 (c). Let $\td g = h g(t)$ and $q \in (n, 2n)$. For any point $x_0 \in M$, let $\td r =d_{\td g}(x_0, x)$ be the geodesic distance from $x_0$ to $x$ under $\td g$. Let
 \[
 \psi = \Delta_{\td g} \td r - \frac{n-1}{\td r}
 \]defined outside the cut-locus of $x_0$.

 There exists a constant $\beta^*_q$ depending only on $q$ and the basic parameters of the initial metric such
 that
 \[
 \Vert \psi_+ \Vert_{L^q(M, \td g)} = \left(\int_M \psi^q_+ d\td g \right)^{1/q} \le \beta^*_q.
 \]
\end{theorem}

One can also replace $M$ in the integral by any proper ball $B(x_0, r, \td g)$ and $\beta^*_q$ by a function depending also on $r$. See Remark \ref{rmth1.2ball} below.

An immediate consequence is the classical volume comparison theorem for both the original and the conformal metric, modulo a small error term which goes to $0$ as the scale goes to $0$ in a definite rate.

\begin{corollary}
\lab{volcomp} Let $g=g(t)$ be as in Theorem \ref{thplapc}.

 (a). Given $q \in (n, 2n)$, there exists a constant $\td b_q$ depending only on $q$ and the basic parameters of the initial metric such
 that
 \[
\frac{d}{dr}\left(\frac{|B(x,r, \td g)|_{\td g}}{r^n}\right)
\le \frac{\td b_q}{r^{n/q}},
\]
\[
\frac{|\partial B(x, r_2, \td g)|_{\td g}}{r^{n-1}_2} -
\frac{|\partial B(x, r_1, \td g)|_{\td g}}{r^{n-1}_1}
 \le \td b_q \left(r^{(q-n)/q}_2 - r^{(q-n)/q}_1 \right).
\]

(b). Given $q \in (n, 2n)$, there exists a constant $b_q$ depending only on $q$ and the basic parameters of the initial metric such
\[
\frac{|B(x, r_2)|}{r^n_2} - \frac{|B(x, r_1)|}{r^n_1} \le b_q \, r^{1-(n/q)}_2, \qquad r_2>r_1>0.
\]
\end{corollary}
We mention that when $q \to n$ or $2n$ the constants $b_q$ and $\td b_q$ may go to $\infty$.

\begin{remark}
\lab{rmk1}
In Theorem C of \cite{CMT:2}, a similar result to part (b) of the corollary is presented. However, the constant on the right hand side goes to infinity as $r_1 \to r_2^-$.
Also, during the proof, on p15, line 10, the stated volume ratio bound seems to require further proof. The reason is that the "time changed" manifold is not the same as those standard weighted manifolds treated in the cited references on the same page. In the time changed manifold, the metric is $e^{2f} g$ and the volume element is $e^{2 f} dg$ which is different from $e^{n f} dg$ in the standard weighted case. The curvature dimension inequality implies certain Laplace comparison theorem. But the latter does not immediately imply the stated bound for the volume ratio for the time changed manifold. The reason is that the required inequality between volume ratio and Laplace comparison
\[
\frac{d}{dr}\left(\frac{|B(x,r)|}{r^n}\right)
\leq \frac{1}{r^{n+1}}\int_{\mb{S}^{n-1}}\int_0^r s\,\left(\Delta r-
\frac{n-1}{r} \right)_+ \, \sqrt{det g(s,\theta)} dsd\theta
\]may not be available for the time changed manifold.
\end{remark}

\begin{remark}
The method used in the proof of the theorem may be applicable to manifolds with Bakry-\'Emery curvature bounded from below under weaker than usual assumptions on the potential function $f$ or $\nb f$. For a general manifold $(M, g)$, in the $K_2$ condition on the Ricci curvature \eqref{defkato}, the power $2$ can be lowered somewhat as far as the Laplace comparison and relative volume comparison are concerned. This may be of independent interest.
\end{remark}

 As an application of Theorem \ref{thplapc}, the second result of the paper confirms the Hamilton-Tian conjecture:

\begin{theorem}
\label{thhtconj}
 As $t \to \infty$,  a subsequence the manifolds (M, g(t)) in \eqref{krf} converge, in $C^\infty_{loc}$ sense,  to a shrinking K\"ahler Ricci soliton with at most real co-dimension 4 singularities.
\end{theorem}

Let us describe further the idea of the proof. As mentioned, the main difficulty is the lack of certain integral Laplace comparison theorem for $(M, g(t))$ due to the lack of suitable Ricci bound. Instead we will consider a conformal metric $\td g= h g(t) \equiv e^{ 2 f(\cdot, t)} g(t)$ where $h$ is determined by an appropriate elliptic equation to cancel the worst part of the Ricci term. The general idea of using conformal metrics is not new and can be found in \cite{Zz2:1} e.g. What is new is the way to handle the extra terms such as $\Delta h$, $\Hess h$ which appear in the Ricci curvature for $\td g$ and suitable a priori bounds for $f$, $|\nb f|$. For example, we will show $f$ is bounded between two uniform constants, $|\nb f|$ is almost bounded. After these, we can prove the desired integral Laplace comparison result. Then we will extend the Cheeger-Colding theory to $(M, \td g(t))$, and prove convergence modulo co-dimension $2$ singularity. Finally using further a priori bounds for the Ricci and curvature tensor for the conformal metrics, we can generalize the result in Cheeger-Colding-Tian \cite{CCT} to conclude that the singular set has co-dimension at least $4$. The a priori bound for $|\nb f|= 0.5 |\nb \ln h|$ in Theorem \ref{thplapc} is just short of $L^\infty$ bound, which makes the generalization more involved. Recall that in \cite{WZ1} Sec. 5, 6, these generalizations and extensions have been done under the slightly stronger assumption that $|\nb f| \in L^\infty$. During the proof, recent work of several authors will also be used and they will be cited in the main body of the paper.

Theorem \ref{thplapc} and Theorem \ref{thhtconj} will be proven in Section 2 and 3 respectively.

Here are some additional notations and conventions to be used frequently. Given a tensor $T$ and metric $g$, $|T|_g$ denotes its Hilbert-Schmidt norm at a point. $C$ with or without index is a positive constant which may change from line to line. G.H. stands for Gromov-Hausdorff topology. $B_{\mathbb{R}^n}(0, r)$ denotes Euclidean balls centered at $0$ with radius $r$. $Rm_g$, $Ric_g$ and $R_g$ denote respectively the curvature tensor, Ricci curvature and scalar curvature with respect to the metric $g$, which may be dropped if no confusion arises.

\section{Proof of Theorem \ref{thplapc}, Laplace comparison}

The proof is divided into four steps.
\medskip

{\it Step 1.}  Converting to a conformal metric.

 Let $h=e^{2f}$ be a positive smooth function on $(M, g)$ and
\be
\lab{deftdg}
\tilde{g} = h g
\ee be a conformal metric. The function $h$, being arbitrary for now, will be determined by an elliptic Schr\"odinger equation involving the negative part of the Ricci curvature for $(M, g)$ later.

In the following, we will use $\Delta$, $\nabla$, $Hess$ and $<\cdot, \cdot>$  to denote
the Laplace, gradient, Hessian and inner products with respect to $g$. The corresponding operators with respect to $\tilde{g}$ will be denoted by $\Delta_{\tilde{g}}$, $\nabla_{\tilde{g}}$, $Hess_{\tilde{g}}$ and $<\cdot, \cdot>_{\tilde{g}}$ respectively. The volume elements for $g$ and $\tilde g$ are denoted by $dg$ and $d \tilde g$. Given a tensor field $X$ on $M$, the notations $|X|$ and $|X|_{\tilde{g}}$ are used for its norms under $g$ and $\tilde g$ respectively. For example, if $X$ is a vector field, then $|X|^2=<X, X>$ and
$|X|^2_{\tilde{g}}=<X, X>_{\td g}$.

Let us recall the transformation rules of those operators under the conformal change with a brief computation.
It is convenient to compute in a local coordinates $\{ x_i \}$ where $g=g_{i j} dx_i \otimes d x_j$ and $\pa_i \equiv \pa_{x_i}$. Let $u$ be a smooth function on $M$. Then
\be
\lab{nabchange}
|\nabla_{\td g} u |^2_{\td g} = \td g^{ij} \pa_i u \pa_j u= h^{-1} g^{ij} \pa_i u \pa_j u = h^{-1} | \nabla u |^2.
\ee Next, since $d \td g= \sqrt{ det \td g} dx_1...dx_n = h^{n/2} \sqrt{ det  g} dx_1...dx_n= h^{n/2} dg$, we have
\[
\al
\Delta_{\td g} &= \frac{1}{h^{n/2} \sqrt{ det  g}} \pa_i \left( h^{n/2} \sqrt{ det  g} \, h^{-1} g^{ij} \pa_j \right) \\
&=\frac{1}{h \sqrt{ det  g}} \pa_i \left(  \sqrt{ det  g} \, g^{ij} \pa_j \right)
+ \frac{n-2}{2} h^{-2} g^{ij} \pa_i h \pa_j.
\eal
\]Therefore,
\be
\lab{lapchange}
\Delta_{\td g} u  = h^{-1} \Delta u + \frac{n-2}{2} h^{-2} < \nb h, \nb u>.
\ee

Next we discuss the transformation of the Hessian of a function. The standard formula for the Hessian in the local coordinates is
\[
(Hess \, u)_{ij} = \pa_i \pa_j  u- \Gamma^k_{ij} \pa_k u
\]where
\[
\Gamma^k_{ij}= \frac{1}{2} g^{kl} (\pa_i g_{jl}+ \pa_j g_{li}- \pa_l g_{ij})
\]is the Christoffel symbol with respect to $g$. Likewise
\[
(Hess_{\td g} \, u)_{ij} = \pa_i \pa_j  u- \td \Gamma^k_{ij} \pa_k u
\]where
\[
\td \Gamma^k_{ij}= \frac{1}{2} \td g^{kl} (\pa_i \td g_{jl}+ \pa_j \td g_{li}- \pa_l \td g_{ij})
\]Since
\[
\td \Gamma^k_{ij}= \frac{1}{2} h^{-1} g^{kl} \left(\pa_i (h g_{jl})+ \pa_j (h g_{li})- \pa_l (h g_{ij}) \right),
\]we deduce
\[
\td \Gamma^k_{ij}= \Gamma^k_{ij} + \frac{1}{2} h^{-1} g^{kl} (\pa_i h g_{jl}+ \pa_j h g_{li}- \pa_l h g_{ij}).
\]Therefore, we have the transformation rule for the Hessian
\be
\lab{hesschange}
(Hess_{\td g} \, u)_{ij} = (Hess \, u)_{ij} - \underbrace{\frac{1}{2} h^{-1} g^{kl} (\pa_i h g_{jl}+ \pa_j h g_{li}- \pa_l h g_{ij}) \pa_k u}_{A_{ij}} \equiv (Hess \, u)_{ij} - A_{ij}.
\ee Since both $(Hess_{\td g} \, u)_{ij}$ and $(Hess \, u)_{ij}$ are $(2, 0)$ tensors, so is $A_{ij}$. In fact
\[
A_{ij} dx_i \otimes dx_j= \frac{1}{2} h^{-1} \left( dh \otimes du + du \otimes dh - <dh, du> g \right).
\]We will use this fact to facilitate computations in local orthonormal coordinates, which shows that there is an absolute constant $c_0$ such that
\be
\lab{normaij}
|(A_{ij})| \le c_0 h^{-1} |\nb h| \,| \nb u|.
\ee For the norms of the Hessians, we have, from \eqref{hesschange}, that
\[
\al
| Hess_{\td g} \, u |^2_{\td g}& = \td g^{ij} \td g^{kl} (Hess_{\td g} \, u)_{ik}
(Hess_{\td g} \, u)_{jl}\\
&= h^{-2}  g^{ij} g^{kl} [ (Hess \, u)_{ik} -A_{ik}] \, [ (Hess \, u)_{jl} -A_{il}].
\eal
\]This implies, for an absolute constant $c_2$, that
\be
\lab{hwan-h}
\left| h^2 | Hess_{\td g} \, u |^2_{\td g} - |Hess \, u |^2 \right | \le c_2  \, h^{-1} \, |\nb h| \,  |\nb u| \, |Hess \, u| + c_2  h^{-2} |\nb h|^2 \,  |\nb u|^2.
\ee

Next, we compute the transformation of the Bochner formula for the gradient norm of a function under $\td g$. Notice that the Laplacian is with respect to the original metric. So the following is not the usual Bochner formula.
\[
\al
\Delta ( |\nb_{\td g} u |^2_{\td g}) &= \Delta ( h^{-1} |\nb u |^2 ) = \Delta ( e^{- 2 f} |\nb u |^2 )\\
&= | \nb u|^2 \Delta e^{-2 f} + 2 <\nb | \nb u |^2, \nb e^{-2f}> + e^{-2f} \Delta |\nb u|^2.
\eal
\]Using the usual Bochner formula for $(M, g)$, we deduce
\be
\lab{boch1}
\al
\Delta ( |\nb_{\td g} u |^2_{\td g}) &=| \nb u|^2 (e^{-2f} 4 |\nb f|^2 - 2 e^{-2f} \Delta f) - 4 <\nb |\nb u|^2, \nb f> e^{-2f}\\
&\qquad + e^{-2f} \left( 2 | Hess \, u|^2 + 2 < \nb \Delta u, \nb u> + 2 Ric (\nb u, \nb u) \right).
\eal
\ee With the property that $Ric \ge -V g$ for the nonnegative function
\be
\lab{defVric-}
V=V(x)=|Ric^-(x)|,
\ee the above implies
\be
\lab{boch1>}
\al
e^{2 f} \Delta ( |\nb_{\td g} u |^2_{\td g}) & \ge | \nb u|^2 ( 4 |\nb f|^2 - 2  \Delta f - 2 V) - 4 <\nb |\nb u|^2, \nb f> \\
&\qquad +  2 | Hess \, u|^2 + 2 < \nb \Delta u, \nb u>.
\eal
\ee The first term on the right hand side is the most dangerous one. We will cancel it out by a suitable choice or $f$, to be done in the next step.

 Note that one can assume that $V$ is smooth without losing generality since one can first choose the function $V_1=|Ric^-(x)| + \epsilon$ and smooth out $V_1$ to a smooth function $V \ge |Ric^-(x)|$.
\medskip

{\it Step 2.} Determining $f$ and hence $h=e^{2f}$.

\noindent {\it Step 2.1.}  The equation for $f$.

Recall from Lemma 2.3 in \cite{TZq2} that $|Ric|^2$ has the a priori bound ($K_2$ bound): for a constant $\kappa_0$, depending only on the main parameters of the initial metric such that
\be
\lab{rick2}
\sup_{x \in M} \int_{M} \frac{|Ric(y)|^2}{d^{n-2}(x, y)} dg(y) \le \kappa_0.
\ee Here we have suppressed the time variable for simplicity.  Pick a number $r \in (0, diam( M, g)/2)$, by Cauchy-Schwarz inequality, we have
\be
\lab{rick1r}
\al
&\sup_{x \in M} \int_{B(x, r)} \frac{V(y)}{d^{n-2}(x, y)} dg(y)\\
&\le \sup_{x \in M} \int_{B(x, r)} \frac{|Ric(y)|}{d^{n-2}(x, y)} dg(y) \le \sqrt{\kappa_0}
\sup_{x \in M} \left(  \int_{B(x, r)} \frac{dg(y)}{d^{n-2}(x, y)}  \right)^{1/2} \le C_1 r \sqrt{\kappa_0}.
\eal
\ee Here $C_1$ depends only on the volume doubling constant. Recall from the $\kappa$ non-collapsing and non-inflating property for the normalized  K\"ahler Ricci flow, there exists a uniform constant $\kappa>0$, depending only on the initial metric such that
\[
\kappa r^n \le |B(x, r)| \le \kappa^{-1} r^n.
\]In other words, the volume of geodesic balls satisfy the Alfors regularity. Recall the embedding result in Lemma 2.4 in \cite{TZq2},

\begin{lemma}
\lab{leimbed}
 Let $V$ be a smooth function on $M$, $p$ be a point on $M$ and $r$ be a
positive number such that $r \le diam(M)/2$. Then for any smooth function $f$
on $M$, the following embedding result holds
\be
\al
&\int_{B(p, r)} |V(x)| f^2(x) dg(x) \\
&\le C \sup_{z \in B(p, 2r)} \int_{B(z, 2r)}
\frac{|V(x)|}{d(z, x)^{n-2}} dg(x) \left( \Vert \nb f \Vert^2_{L^2(B(p, 2r)}
+ r^{-2} \Vert  f \Vert^2_{L^2(B(p, 2r)} \right).
\eal
\ee Here $C$ is a positive constant depending only on the basic parameters of $g(0)$.
\end{lemma}

Due to its importance to us, for completeness we provide a

\proof (of Lemma \ref{leimbed}).

With the gradient bound for the Green's function and volume lower
and upper bound given in the basic properties, the lemma and its proof are essentially
known in the literature. See \cite{Si:1} and \cite{CGL:1} for
example. The only difference is that we are dealing with compact
manifolds which create one extra term from the Green's formula.

Without loss of generality we assume $V \ge 0$.
Let $\phi$ be a smooth cut-off function in $B(p, 2r)$ such that $\phi =1$ on
$B(p, r)$ and that $|\nb \phi | \le C/r$. Then from the Green's formula, we have
\be
f\phi(x) - ave (f \phi) = - \int \Gamma(x, y) \Delta (f \phi)(y) dg(y)
\ee where $ave (f \phi)$ is the average of $f \phi$ over $M$, and $\Gamma$ is the
Green's function. After integration by parts, this becomes
\be
f\phi(x) = ave (f \phi) + \int \nb_y \Gamma(x, y) \nb (f \phi)(y) dg(y).
\ee By the bound for $\nb \Gamma$ in (\ref{GdGjie}), this infers
\be
|f \phi(x)| \le |ave (f \phi)| + \int \frac{C |\nb f| \phi(y)}{d(x, y)^{n-1}}  dg(y)
+ \int \frac{C | f \nb \phi(y)|}{d(x, y)^{n-1}}  dg(y).
\ee

For any smooth test function $\eta \ge 0$ supported in $B(p, r)$, we have, from the previous inequality,
\be
\lab{Vfi123}
\al
&\int \sqrt{V} f \eta(x) dg(x)  \\
&\le C \int \sqrt{V} \eta (x) \int \frac{ |\nb f| \phi(y)}{d(x, y)^{n-1}}  dg(y) dg(x) + C \int \sqrt{V} \eta (x) \int \frac{ | f \nb \phi(y)|}{d(x, y)^{n-1}}  dg(y) dg(x) \\
&\qquad  +  \int \sqrt{V} \eta (x) ave ( f \phi) dg(x)\\
&\equiv I_1 + I_2 +I_3. \eal
\ee
Following the argument in
\cite{Si:1} and \cite{CGL:1}, one can show that
\be \lab{Vfi1}
I^2_1 \le C \Vert \nb f \Vert^2_{L^2(B(p, 2r))} \Vert \eta
\Vert^2_{L^2(B(p, r)} \sup_{z \in B(p, 2r)} \int \frac{V(x)}{d(x,
z)^{n-2}} dg(x). \ee \be \lab{Vfi2} I^2_2 \le \frac{C}{r^2} \Vert
\nb f \Vert^2_{L^2(B(p, 2r))} \Vert \eta \Vert^2_{L^2(B(p, r)}
\sup_{z \in B(p, 2r)} \int \frac{V(x)}{d(x, z)^{n-2}} dg(x). \ee
The idea for estimating $I_1$ is to use the Fubini theorem first and Cauchy-Schwarz inequality second  and the Fubini theorem again after converting the $L^2$ norm of a function as a double integral. Here is the detail.
\[
\al
&I^2_1 \le C \Vert \nb f \Vert^2_{L^2(B(p, 2r))} \int_{B(p, 2r)} \left( \int_{B(p, 2r)}
  \frac{\sqrt{V(x)} \eta (x) }{d(x, y)^{n-1}} dg(x) \right)^2 dg(y)\\
&=C \Vert \nb f \Vert^2_{L^2(B(p, 2r))} \int_{B(p, 2r)} \int_{B(p, 2r)} \int_{B(p, 2r)}
  \frac{\sqrt{V(z)} \eta (z) }{d(z, y)^{n-1}}  \frac{\sqrt{V(x)} \eta (x) }{d(x, y)^{n-1}} dg(x) dg(z) dg(y)\\
&=C \Vert \nb f \Vert^2_{L^2(B(p, 2r))} \\
& \qquad \times \int_{B(p, 2r)} \int_{B(p, 2r)}
\sqrt{V(z)} \eta (z) \sqrt{V(x)} \eta (x) \int_{B(p, 2r)}
  \frac{1}{d(z, y)^{n-1} \, d(x, y)^{n-1}} dg(y) dg(x) dg(z).
\eal
\]Using basic properties 4 and 5 for volumes of geodesic balls, it is straight forward to deduce
\[
\int_{B(p, 2r)}
  \frac{1}{d(z, y)^{n-1} \, d(x, y)^{n-1}} dg(y) \le \frac{C}{d^{n-2}(x, z)}
\]where $C$ depends only on the diameter of $M$ and the volume non-collapsing and non-inflating constant. Combining the last two inequalities, we infer
\[
\al
&I^2_1 \le
C^2 \Vert \nb f \Vert^2_{L^2(B(p, 2r))}  \int_{B(p, 2r)} \int_{B(p, 2r)}
\frac{\sqrt{V(z)} \eta (z)}{d(x, z)^{(n-2)/2}}  \frac{\sqrt{V(x)} \eta (x)}{d(x, z)^{(n-2)/2}} dg(x) dg(z)\\
&\le
C^2 \Vert \nb f \Vert^2_{L^2(B(p, 2r))}  \left(\int_{B(p, 2r)} \int_{B(p, 2r)}
\frac{V(z) \eta^2 (z)}{d(x, z)^{(n-2)}} dg(x) dg(z) \right)^{1/2}\\
&\qquad \times
\left(\int_{B(p, 2r)} \int_{B(p, 2r)}
\frac{V(x) \eta^2 (x)}{d(x, z)^{(n-2)}} dg(x) dg(z) \right)^{1/2}
 \eal
\]This implies the estimate for $I_1$ in \eqref{Vfi1}  after exchanging the order of integration again. The bound for $I_2$ is easily obtained by Cauchy-Schwarz inequality since the singularity is cut off.

Next \be \al
|I_3| &\le \int \sqrt{V} |\eta(x)| \frac{1}{|M|} \int |f| \phi(y) dg(y) dg(x)\\
&\le \int \sqrt{V} |\eta(x)| \frac{diam(M)^{n-1}}{|M|} \int \frac{|f| \phi(y) dg(y)}{d(x, y)^{n-1}} dg(x)\\
&\le C \int \sqrt{V} |\eta(x)|  \int \frac{|f| \phi(y) dg(y)}{d(x, y)^{n-1}} dg(x),
\eal
\ee where we have used Perelman's bound on the diameter and volume non-collapsing property.  Just like the case for $I_1$, we can now deduce
\be
\lab{Vfi3}
|I_3|^2 \le C \Vert  f \Vert^2_{L^2(B(p, 2r))} \Vert \eta \Vert^2_{L^2(B(p, r)}
\sup_{z \in B(p, 2r)}
\int \frac{V(x)}{d(x, z)^{n-2}} dg(x).
\ee Substituting (\ref{Vfi1}), (\ref{Vfi2}) and (\ref{Vfi3}) into (\ref{Vfi123}),
we find that
\be
\al
&\int \sqrt{V} f \eta(x) dg(x) \\
&\le
C \left[\Vert \nb f \Vert_{L^2(B(p, 2r))}  + \frac{1+r}{r} \Vert  f \Vert_{L^2(B(p, 2r))} \right] \Vert \eta \Vert_{L^2(B(p, r))}
\left(\sup_{z \in B(p, 2r)}
\int \frac{V(x)}{d(x, z)^{n-2}} dg(x) \right)^{1/2}.
\eal
\ee Since $\eta$ is arbitrary, the result follows by applying the Riesz theorem.

\qed

This embedding, \eqref{rick1r} and a simple covering argument using balls with a small but definite radius, we deduce
\be
\lab{quadrav}
\int_M \left( |\nb \phi|^2-V \phi^2 \right) dg \ge - \rho_0 \int_M \phi^2 dg, \qquad \forall \phi \in C^\infty(M).
\ee Here $\rho_0 \ge 0 $ is a constant depending only on the main parameters of the initial metric. See also \eqref{vsubcr} below. In the above, we have used the fact that the diameters of $(M, g(t))$ are uniformly bounded due to Perelman c.f. \cite{ST:1}. In another word, the quadratic form on the left hand side is bounded from below.
Let us choose $\rho_0$ to be the optimal one such that \eqref{quadrav} holds, i.e.
\be
\lab{defrho0}
\rho_0= -\inf \{ \int_M \left( |\nb \phi|^2-V \phi^2 \right) dg \, | \, \phi \in C^\infty(M), \, \Vert \phi \Vert_{L^2(M)}=1 \}.
\ee The standard variation theory shows that there exists a smooth positive function $\phi$, $\Vert \phi \Vert_{L^2}=1$,  for which the inequality becomes equality and $\phi$ satisfies the following equation on $M$:
\be
\lab{eqphi}
\Delta \phi + V \phi - \rho_0 \phi =0.
\ee

Define
\be
\lab{deffphi}
f= \ln \phi.
\ee From \eqref{eqphi}, we see that $f$ satisfies the equation
\be
\lab{eqffphi}
\Delta f + |\nb f|^2 + V - \rho_0 = 0.
\ee Substituting this to the first term on the right hand side of
\eqref{boch1>}, we deduce a modified curvature dimension inequality
\be
\lab{modcdineq}
\al
e^{2 f} \Delta ( |\nb_{\td g} u |^2_{\td g})  \ge  2 | Hess \, u|^2 + 2 < \nb \Delta u, \nb u>   - 4 <\nb |\nb u|^2, \nb f> - 2 \rho_0 | \nb u|^2.
\eal
\ee Notice that the most dangerous term involving the Ricci curvature in the Bochner formula is eliminated.

Recalling that $h=e^{2 f}$, we have determined that
\be
\lab{defh=}
h=\phi^2.
\ee In order to proceed, we need to know more about the function $h$. This is done in the next sub-step.

\medskip
\noindent {\it Step 2.2.} Properties for $h=e^{2 f}=\phi^2$.

Although for each time slice, the function $\phi$ is smooth, what we really need are uniform regularities for $\phi$, which are independent of time. We present them in the next two propositions. Again we have suppressed the time variable.

\begin{proposition}
\lab{prphijie}
(a). There exist two positive constants $b_1, b_2$, depending only on the main parameters of the initial metric such that
\[
b_1 \le h(x)= \phi^2(x) \le b_2, \qquad \forall x \in M.
\]

(b). There exists and $b_3>0$,  depending only on the main parameters of the initial metric such that
\[
\Vert \phi \Vert_{C^{1/2}(M)} \le b_3.
\]

\lab{prhphi1}

\end{proposition}
\proof (a). The upper bound is a consequence of equation \eqref{eqphi}, the normalization $\Vert \phi \Vert_{L^2}=1$ and Moser's iteration. Note that the embedding Lemma 2.4 in \cite{TZq2} (Lemma \ref{leimbed} here) and \eqref{rick1r} imply that $V$, regarded as a potential function, is dominated  by $\Delta$ in the following sense: there exists a fixed power $p* \ge 1$ and $C_* >0$, depending only on the basic parameters of the initial metric, such that, for every $\e>0$,
\be
\lab{vsubcr}
 \int_M  V  \eta^2 dg \le \e \int |\nb \eta|^2 dg + C_* \e^{-p_*} \int_M \eta^2 dg, \qquad \forall \eta \in C^\infty(M).
\ee This can be done by choosing $r$ in that lemma as a suitable power of $\e$. In addition, we also have uniform Sobolev embedding theorem from the basic properties listed in the introduction. So Moser's iteration can be carried out in a routine manner.

The lower bound can be obtained by first proving a Harnack inequality for $\phi$ and using the upper bound and the uniform boundedness of the diameters. The proof of the Harnack inequality is more or less standard but tedious due to the presence of $V$. So we give a direct proof as follows.

Since $h=e^{2 f}$, it is enough to prove that $f$ has a definite lower bound depending only on the basic parameters. By \eqref{eqffphi},
\[
\Delta f = - |\nb f|^2 - V + \rho_0  \le \rho_0- V.
\]Let $D$ be the domain where $f<0$ and $f_-=- \min\{f(x), 0 \}$. If $D$ is empty, then we are done with the proof since $f \ge 0$ on $M$. So we assume $D$ is not empty. Then
\be
\lab{eqf-}
- \Delta f_-(x) \le \rho_0- V(x), \qquad x \in D, \qquad f_-(x)=0, \quad x \in \pa D.
\ee Multiplying \eqref{eqf-} by $f_-$ and integrating by parts, we deduce
\[
\al
\int_M |\nb f_-|^2 dg &\le \int_{D} (\rho_0- V(x)) f_-(x) dg \le \Vert \rho_0 -V \Vert_{L^2(M)} \, \Vert f_- \Vert_{L^2(D)} \\
&\le C diam(M) \Vert \rho_0 -V \Vert_{L^2(M)} \, \Vert \nb f_- \Vert_{L^2(D)},
\eal
\]where we have used the uniform Poincar\'e inequality in the basic properties.
Therefore
\[
\Vert \nb f_- \Vert_{L^2(M)} \le C diam(M) \Vert \rho_0 -V \Vert_{L^2(M)}.
\]Then the  uniform Poincar\'e inequality in turn implies that
\[
\Vert  f_- \Vert_{L^2(M)} \le C^2 diam(M)^2 \, \Vert \rho_0 -V \Vert_{L^2(M)}.
\]With this $L^2$ bound in hand, we can use Moser's iteration as described in the previous paragraph on inequality \eqref{eqf-} to deduce that
\[
\Vert  f_- \Vert_{L^\infty(M)} \le C  \Vert  f_- \Vert_{L^2(M)} \le C^3 diam(M)^2 \, \Vert \rho_0 -V \Vert_{L^2(M)}.
\]Again the constant $C$ may have changed but it still only depends on the basic parameters. This proves the lower bound for $f$ and hence $\phi$.

(b). We use \eqref{eqphi} and the Green's function $\Gamma=\Gamma(x, y)$ to write
\[
\phi(x) = \int_M \Gamma(x, y) (V(y)-\rho_0) \phi(y) dg(y).
\]Then, one can use the bound for $\Gamma$ and $\nb \Gamma$ to derive that $\phi \in C^{1/2}$. We use an idea from \cite{CFG} p423-424.

Pick $x, x_0 \in M$ such that $d(x, x_0)<1/10^2$ and let $N>10$ be a large number to be determined later. Then
\be
\lab{phix-x0}
\al
&|\phi(x)-\phi(x_0)| \le \int_M |\Gamma(x, y)-\Gamma(x_0, y)| |V(y)-\rho_0| dg(y)\\
&=\int_{d(x_0, y)>N d(x, x_0)} ... dg(y)
+ \int_{d(x_0, y) \le N d(x, x_0)} ... dg(y) \equiv I_1 + I_2.
\eal
\ee In case ${d(x_0, y)>N d(x, x_0)}$, using the gradient bound for $\Gamma$ in the basic properties and the choice that $d(x_0, y) >> d(x_0, x)$, we have
\[
|\Gamma(x, y)-\Gamma(x_0, y)| \le  \frac{C d(x, x_0)}{d(x_0, y)^{n-1}} =
\frac{C d(x, x_0)}{d(x_0, y) \, d(x_0, y)^{n-2}}
\le  \frac{C}{N d(x_0, y)^{n-2}}.
\]Therefore
\be
\lab{i1c1k}
\al
I_1 &\le \frac{C}{N} \int_M \frac{|V(y)-\rho_0| }{d(x_0, y)^{n-2}} dg(y)
\le
\frac{C}{N} \left( \int_M \frac{|V(y)-\rho_0|^2 }{d(x_0, y)^{n-2}} dg(y) \right)^{1/2}
\, \left( \int_M \frac{1}{d(x_0, y)^{n-2}} dg(y) \right)^{1/2}\\
&\le C_1 \frac{C}{N} K^{1/2}(|V-\rho_0|^2).
\eal
\ee Here we have used the volume properties among the basic properties.

For $I_2$ we use the bound for $\Gamma$ in the basic properties to deduce
\[
\al
I_2 &\le C \int_{d(x_0, y) \le N d(x, x_0)}  \frac{|V(y)-\rho_0| }{d(x_0, y)^{n-2}} dg(y)
+ C \int_{d(x_0, y) \le N d(x, x_0)}  \frac{|V(y)-\rho_0| }{d(x, y)^{n-2}} dg(y)\\
& \le C \left( \int_{d(x_0, y) \le N d(x, x_0)} \frac{|V(y)-\rho_0|^2 }{d(x_0, y)^{n-2}} dg(y) \right)^{1/2}
\, \left( \int_{d(x_0, y) \le N d(x, x_0)} \frac{1}{d(x_0, y)^{n-2}} dg(y) \right)^{1/2}\\
&\qquad + C \left( \int_{d(x_0, y) \le N d(x, x_0)} \frac{|V(y)-\rho_0|^2 }{d(x, y)^{n-2}} dg(y) \right)^{1/2}
\, \left( \int_{d(x_0, y) \le N d(x, x_0)} \frac{1}{d(x, y)^{n-2}} dg(y) \right)^{1/2}.
\eal
\]By the volume properties again, we find, as in the case of \eqref{rick1r},
\be
\lab{i2c2k}
I_2 \le C_2 C K^{1/2}(|V-\rho_0|^2)  N d(x, x_0).
\ee
Substituting \eqref{i2c2k} and \eqref{i1c1k} into \eqref{phix-x0}, we see that
\[
|\phi(x)-\phi(x_0)| \le \frac{C_3}{N} + C_4 N d(x, x_0).
\]This implies, with $N=d(x, x_0)^{-1/2}$, that
\[
|\phi(x)-\phi(x_0)| \le C_5 d(x, x_0)^{1/2},
\]proving the proposition after renaming the constant.

 Here, all the constants only depend on the basic parameters of the original metric.

\qed

\begin{proposition}
\lab{prhphi2}
 For any $p \ge 10$, there exists $b_p>0$,
depending only on the basic parameters of the initial metric and $p$ such that
\[
\Vert \nb \phi \Vert_{L^p(M)} \le b_p.
\]
\end{proposition}

\proof
Denote $V-\rho_0$ by $W$. From \eqref{eqphi}, $\phi$ is a solution to the following equation.
\be
\lab{eqphi2}
\Delta \phi + W \phi =0.
\ee

First we show that
\be
\lab{phil2}
\Vert \nb \phi \Vert_{L^2(M)} \le b_2
\ee by using $\phi$ as a test function in the above equation, which yields
\[
\int_M |\nb \phi |^2 dg = \int_M (V-\rho_0) \phi^2 dg.
\]Then \eqref{phil2} follows from the embedding inequality \eqref{vsubcr} with $\e=1/2$ and Proposition \ref{prphijie} (a). For the higher integrability we will apply Moser's iteration on the equation for $|\nb \phi|^2$.

 By the Bochner formula
\be
\lab{bochphi}
\al
\Delta |\nb \phi|^2 &= 2 | Hess \, \phi |^2 + 2 < \nb \Delta \phi, \nb \phi> + 2 Ric(\nb \phi, \nb \phi)\\
&\ge - 2 < \nb  (W \phi), \nb \phi> - 2 V |\nb \phi|^2.
\eal
\ee Pick $p \ge 1$, using $(|\nb \phi|^2)^{2p+1}$ as a test function in the above inequality, we find
\[
\al
(2p+1) &\int_M < \nb |\nb \phi|^2, \nb |\nb \phi|^2>  (|\nb \phi|^2)^{2p} dg \\
&\le 2 \int_M  < \nb  (W \phi), \nb \phi> (|\nb \phi|^2)^{2p+1} dg + 2  \int_M V (|\nb \phi|^2)^{2p+2} dg.
\eal
\]After integration by parts again, the above implies
\be
\lab{caciodh}
\al
&\frac{2p+1}{(p+1)^2} \int_M |\nb (|\nb \phi|^2)^{p+1} |^2 dg\\
&\le -2 \int_M  (W \phi) \Delta \phi (|\nb \phi|^2)^{2p+1} dg
-2 \int_M W \phi <\nb \phi, \nb (|\nb \phi|^2)^{2p+1}> dg\\
&\qquad
+ 2  \int_M V (|\nb \phi|^2)^{2p+2} dg\\
&= \underbrace{2 \int_M  (W \phi)^2 (|\nb \phi|^2)^{2p+1} dg}_{T_1} \,
\underbrace{-2 (2p+1) \int_M W \phi <\nb \phi, \nb (|\nb \phi|^2)>  (|\nb \phi|^2)^{2p} dg}_{T_2}\\
&\qquad
+ \underbrace{ 2  \int_M V (|\nb \phi|^2)^{2p+2} dg}_{T_3}.
\eal
\ee In the following, we bound $T_3, T_1$ and $T_2$ respectively. Unless stated otherwise, all constants depend only on the main parameters of the initial metric.

From Lemma 2.4 in \cite{TZq2}, the second part, we have, $\forall \e>0$,
\[
\al
T_3 &\le \e \int_M V^2 (|\nb \phi|^2)^{2p+2} dg + \e^{-1} \int_M  (|\nb \phi|^2)^{2p+2} dg\\
&\le C_1 K(V^2) \e  \int_M  |\nb (|\nb \phi|^2)^{p+1}|^2 dg + (\e^{-1} + C_2 \e K(V^2)) \int_M  (|\nb \phi|^2)^{2p+2} dg,
\eal
\]where $K(V^2)$ is defined in item 9 of the basic properties in the introduction.
Choosing $\e =(2p+1)/[4(p+1)^2 C_1 K(V^2)]$, we deduce
\be
\lab{phit3<}
\al
T_3 &\le \frac{2p+1}{4(p+1)^2} \int_M |\nb (|\nb \phi|^2)^{p+1} |^2 dg
+ C_3 (1+p )K(V^2) \int_M  (|\nb \phi|^2)^{2p+2} dg.
\eal
\ee

For $T_1$, similarly, we have
\[
\al
T_1 &\le 4 \Vert \phi \Vert_\infty \int_M (V^2 + \rho^2_0) (|\nb \phi|^2)^{2p+1} dg\\
&\le 4 \Vert \phi \Vert_\infty \int_M (V^2 + \rho^2_0) [\e (|\nb \phi|^2)^{2p+2}
+ \e^{-2p-1}] dg\\
&\le C_1 K(V^2) \Vert \phi \Vert_\infty \e  \int_M  |\nb (|\nb \phi|^2)^{p+1}|^2 dg + \e [C_2 K(V^2) + 4 \rho_0^2] \Vert \phi \Vert_\infty \int_M  (|\nb \phi|^2)^{2p+2} dg \\
&\qquad + 4 \Vert \phi \Vert_\infty  \e^{-2p-1} \int_M (V^2 + \rho^2_0) dg.
\eal
\]This implies, with $\e =(2p+1)/[4(p+1)^2 C_1 \Vert \phi \Vert_\infty K(V^2)]$, that
\be
\lab{phit1<}
\al
T_1 &\le \frac{2p+1}{4(p+1)^2} \int_M |\nb (|\nb \phi|^2)^{p+1} |^2 dg\\
&\qquad + C_4 \left(1+\Vert \phi \Vert_\infty [(1+p )K(V^2) \Vert \phi \Vert_\infty ]^{2p+1} \right) \left( \int_M  (|\nb \phi|^2)^{2p+2} dg + 1 \right).
\eal
\ee

Finally,
\[
\al
|T_2| &\le 2 (2p+1) \Vert \phi \Vert_\infty \int_M |W|  \,  |\nb (|\nb \phi|^2)| \, (|\nb \phi|^2)^{2p+ (1/2)} dg\\
&=\frac{2 (2p+1)}{p+1} \Vert \phi \Vert_\infty \int_M |W| \,   |\nb (|\nb \phi|^2)^{p+1}| \, (|\nb \phi|^2)^{p+ (1/2)} dg\\
&\le \frac{2p+1}{8(p+1)^2} \int_M |\nb (|\nb \phi|^2)^{p+1} |^2 dg
+ 16 (2p+1) \Vert \phi \Vert_\infty^2 \int_M W^2 (|\nb \phi|^2)^{2p+1} dg.
\eal
\]The last term can be bounded as for $T_1$ with $\e$ suitably adjusted, giving us
\be
\lab{phit2<}
\al
T_2 &\le \frac{2p+1}{4(p+1)^2} \int_M |\nb (|\nb \phi|^2)^{p+1} |^2 dg\\
&\qquad + C_5 \left(1+\Vert \phi \Vert_\infty^2 [(1+p^2 )K(V^2) \Vert \phi \Vert_\infty^2 ]^{4p^2+2} \right) \left( \int_M  (|\nb \phi|^2)^{2p+2} dg + 1 \right).
\eal
\ee

Substituting \eqref{phit1<}, \eqref{phit2<} and \eqref{phit3<} into \eqref{caciodh}, we arrive at
\be
 \int_M |\nb F_p |^2 dg \le C_6 p  \left(1+\Vert \phi \Vert_\infty^2 [(1+p^2 )K(V^2) \Vert \phi \Vert_\infty^2 ]^{4p^2+2} \right) \left( \int_M  F_p^2 dg + 1 \right),
\ee where we have written
\[
F_p \equiv (|\nb \phi|^2)^{p+1}.
\]Let $\overline{F}_p = \int_M F_p dg$, then the uniform Sobolev embedding theorem in the basic properties implies
\[
\al
&\left(\int_M |F_p - \overline{F}_p|^{2n/(n-2)} dg \right)^{(n-2)/n} \le C \int_M |\nb F_p |^2 dg \\
&\le C
C_6 p  \left(1+\Vert \phi \Vert_\infty^2 [(1+p^2 )K(V^2) \Vert \phi \Vert_\infty^2 ]^{4p^2+2} \right) \left( \int_M  F_p^2 dg + 1 \right).
\eal
\]Hence
\be
\lab{mositF}
\al
&\left(\int_M F_p^{2n/(n-2)} dg \right)^{(n-2)/n} \\
&\le C
C_6 p  \left(1+\Vert \phi \Vert_\infty^2 [(1+p^2 )K(V^2) \Vert \phi \Vert_\infty^2 ]^{4p^2+2} \right) \left( \int_M  F_p^2 dg + 1 \right) + |M|^{(n-2)/n} \overline{F}_p.
\eal
\ee Since the power on $|\nb \phi|$ on  the left hand side is higher, from here, Moser's iteration and \eqref{phil2} imply that
\[
\Vert \nb \phi  \Vert_{L^p(M)} \le b_p.
\]Since $h=\phi^2$, this completes the proof of the proposition and parts (a), (b) of the theorem.
\qed

\medskip

{\it Step 3.}  Applying the conformal distance $\td r$ on \eqref{modcdineq}, the modified curvature dimension inequality.

According to \eqref{nabchange} and \eqref{lapchange},
\[
\al
&|\nb u|^2= h |\nb_{\td g} u |^2_{\td g}, \\
&  \Delta u = h \Delta_{\td g} u - \frac{n-2}{2} h^{-1} < \nb h, \nb u>=
h \Delta_{\td g} u - \frac{n-2}{2} <\nb_{\td g} h, \nb_{\td g} u>_{\td g}.
\eal
\]Substituting these and \eqref{hesschange} to the right hand side of \eqref{modcdineq} and computing at the center of an orthonormal frame with respect to $g$, we deduce
\be
\lab{modcdineq2}
\al
&\frac{1}{2} e^{2 f} \Delta ( |\nb_{\td g} u |^2_{\td g})  \\
&\ge   \sum^{n}_{i, j=1} | (Hess_{\td g}  \, u)_{i j}+ A_{ij}|^2
+ h \left< \nb_{\td g} \left(h \Delta_{\td g} u - \frac{n-2}{2} <\nb_{\td g} h, \nb_{\td g} u>_{\td g} \right), \nb_{\td g} u \right>_{\td g}  \\
 &\quad - 2 h \left<\nb_{\td g} (h |\nb_{\td g} u |^2_{\td g}), \nb_{\td g} f \right>_{\td g} -  \rho_0 h | \nb_{\td g} u|^2_{\td g}.
\eal
\ee This is main curvature dimension inequality we will work with in the rest of the section.

Fix a point $p_0 \in M$, for another point $x \in M$, we write
\be
\lab{rtilde}
\td r = d_{\td g} (p_0, x).
\ee If $x$ is not in the cut locus of $p_0$, then we can take $u=\td r$ in \eqref{modcdineq2} to obtain, since  $|\nb_{\td g} \td r |_{\td g}=1$, that
\be
\lab{rcdineq1}
\al
&0 \ge   \sum^{n}_{i, j=1} | (Hess_{\td g}  \, \td r)_{i j}+ A_{ij}|^2
+ h \left< \nb_{\td g} \left(h \Delta_{\td g} \td r - \frac{n-2}{2} <\nb_{\td g} h, \nb_{\td g} \td r >_{\td g} \right), \nb_{\td g} \td r \right>_{\td g}  \\
 &\quad - 2 h \left<\nb_{\td g} h, \nb_{\td g} f \right>_{\td g} -  \rho_0 h.
\eal
\ee Note that we can pick an orthonormal frame $\{ e_i \}$ with respect to $g$, whose center is $x$ and that $\pa_{\td r}$ direction is parallel to the $e_1$ direction.
Note that the frame is still orthogonal with respect to $\td g$ since $\td g = h g$.
After this arrangement, we can compute at $x$ to reach that
\[
\sum^{n}_{i, j=1} | (Hess_{\td g}  \, \td r)_{i j}+ A_{ij}|^2
= h^2 \sum^{n}_{i, j=1} h^{-2} | (Hess_{\td g}  \, \td r)_{i j}+ A_{ij}|^2
= h^2 |Hess_{\td g}  \, \td r + A|^2_{\td g},
\]where $A$ is the $(2, 0)$ tensor represented by $A_{ij}$. Therefore
\[
\al
\sum^{n}_{i, j=1} | (Hess_{\td g}  \, \td r)_{i j}+ A_{ij}|^2
&\ge \frac{h^2}{n-1} \left(\Delta_{\td g} \td r + tr_{\td g} A' \right)^2\\
&=\frac{h^2}{n-1} \left(\Delta_{\td g} \td r + h^{-1} \sum^{n}_{i=2} A_{ii} \right)^2
\eal,
\]where $A'$ is the tensor represented by the truncated matrix $(A_{ij})^{n}_{i, j=2}$ and $tr_{\td g} A'$ is its trace with respect to $\td g$. Substituting the above to
\eqref{rcdineq1} and noticing that
\[
2 h \left<\nb_{\td g} h, \nb_{\td g} f \right>_{\td g} = | \nb_{\td g} h |^2_{\td g},
\]we arrive at
\be
\lab{rcdineq2}
\al
&0 \ge \frac{1}{n-1} \left(\Delta_{\td g} \td r + \a \right)^2
+ h^{-1} \pa_{\td r} \left(h \Delta_{\td g} \td r - \frac{n-2}{2} \pa_{\td r} h \right) - h^{-2} | \nb_{\td g} h |^2_{\td g} -  \rho_0 h^{-1}, \\
&\a \equiv h^{-1} \sum^{n-1}_{i=1} A_{ii}.
\eal
\ee Since $A$ is a tensor, from \eqref{hesschange}, we can compute  at the center of an orthonormal coordinates with respect to $g$,
\[
A_{ii}= \frac{1}{2} h^{-1} g^{kl} (\pa_i h g_{il}+ \pa_i h g_{li}- \pa_l h g_{ii}) \pa_k \td r,
\]so that
\be
\lab{afajie}
|\a| \le  c h^{-1} |\nb h | \, |\nb \td r| = c  |\nb_{\td g} h |_{\td g} \, |\nb_{\td g} \td r|_{\td g}= c  |\nb_{\td g} h |_{\td g}.
\ee Here $c$ is a dimensional constant. In the above, we have used the relation
\be
\lab{nabgtdg}
|\nb \cdot |^2 = h |\nb_{\td g} \cdot |^2_{\td g}.
\ee

Let us introduce the function
\be
\lab{defpsi}
\psi= \Delta_{\td g} \td r - \frac{n-1}{\td r}.
\ee Then
\be
\lab{dra2}
\left(\Delta_{\td g} \td r + \a \right)^2=\left(\psi + \frac{n-1}{\td r} + \a \right)^2=
\psi^2 + 2 (\frac{n-1}{\td r} + \a) \psi +\left(\frac{n-1}{\td r} + \a \right)^2.
\ee Also
\be
\lab{h-1pr}
\al
&h^{-1} \pa_{\td r} \left(h \Delta_{\td g} \td r - \frac{n-2}{2} \pa_{\td r} h \right)
=h^{-1} \pa_{\td r} \left(h \left(\psi + \frac{n-1}{\td r}  - \frac{n-2}{2} h^{-1} \pa_{\td r} h \right) \right)\\
&= \pa_{\td r} \left(\psi + \frac{n-1}{\td r}  - \frac{n-2}{2} h^{-1} \pa_{\td r} h \right)
+h^{-1} \pa_{\td r} h \left(\psi + \frac{n-1}{\td r}   - \frac{n-2}{2} h^{-1} \pa_{\td r} h \right)\\
&= \pa_{\td r} (\psi - \beta) - \frac{n-1}{\td r^2}
+ \frac{2}{n-2} \beta \left(\psi + \frac{n-1}{\td r}   - \beta \right).
\eal
\ee Here
\be
\lab{defbeta}
\beta \equiv \frac{n-2}{2} h^{-1} \pa_{\td r} h.
\ee

Plugging \eqref{dra2} and \eqref{h-1pr} into \eqref{rcdineq2}, we find
\[
\al
&0 \ge \frac{1}{n-1} \left(\psi^2 + 2 (\frac{n-1}{\td r} + \a) \psi +\left(\frac{n-1}{\td r} + \a \right)^2 \right)
\\
&
\qquad +\pa_{\td r} (\psi - \beta) - \frac{n-1}{\td r^2}
+ \frac{2}{n-2} \beta  \left(\psi + \frac{n-1}{\td r}   -\beta  \right)
 - h^{-2} | \nb_{\td g} h |^2_{\td g} -  \rho_0 h^{-1}.
\eal
\] Therefore we can cancel the terms involving $1/\td r^2$ to reach
\[
\al
&0 \ge \frac{1}{n-1} \left(\psi^2 + 2 (\frac{n-1}{\td r} + \a) \psi + 2 \frac{n-1}{\td r} \a + \a^2 \right)
\\
&
\qquad +\pa_{\td r} (\psi - \beta)
+\frac{2}{n-2} \beta \left(\psi + \frac{n-1}{\td r}   - \beta \right)
 - h^{-2} | \nb_{\td g} h |^2_{\td g} -  \rho_0 h^{-1}.
\eal
\]Hence, we deduce a Riccati type inequality for $\psi$:
\be
\lab{ricatpsi}
\al
&\frac{1}{n-1} \psi^2 +  \pa_{\td r} (\psi - \beta)+ \frac{2 \psi}{\td r} + \left(\frac{2 \a }{n-1}  + \frac{2 \beta}{n-2} \right) \psi + \left( 2 \a + \frac{2(n-1) \beta}{n-2} \right) \frac{1}{\td r}  \\
&\le \frac{2 \beta^2}{n-2} + h^{-2} | \nb_{\td g} h |^2_{\td g} +  \rho_0 h^{-1}.
\eal
\ee
\medskip

{\it Step 4.}  Integral estimates for $\psi$.

In this step we follow the idea of \cite{PeWe1} to obtain a $L^q$ bound for $\psi_+$ with $q \in (n, 2 n)$. Due to the extra terms in \eqref{ricatpsi}, the computation depends on the $L^\infty$ bound for $h$ and the $L^p$ bound for $|\nb h|$ with large $p$ in Step 2.

Let us pick a number $q \in (n-2, 2n-2)$. Multiplying \eqref{ricatpsi} by $(\psi-\beta)^q_+ \sqrt{det \td g}$  and integrating along a geodesic starting from the point $p_0$ and ending at the cut-locus, we find
\be
\lab{psint1}
\al
& \frac{1}{n-1} \int^{r_*}_0 \psi^2 (\psi-\beta)^q_+ \sqrt{det \td g} d\td r + \underbrace{ \int^{r_*}_0 \pa_{\td r} (\psi - \beta) (\psi-\beta)^q_+ \sqrt{det \td g} d\td r}_{\equiv L_2} \\
&\qquad + \underbrace{\int^{r_*}_0 \frac{2 \psi}{\td r} (\psi-\beta)^q_+ \sqrt{det \td g} d\td r }_{\equiv L_3} + \int^{r_*}_0 \left(\frac{2 \a }{n-1}  + \frac{2 \beta}{n-2} \right) \psi  (\psi-\beta)^q_+ \sqrt{det \td g} d\td r \\
&\qquad + \int^{r_*}_0 \left( 2 \a + \frac{2(n-1) \beta}{n-2} \right) \frac{1}{\td r}  (\psi-\beta)^q_+ \sqrt{det \td g} d\td r \\
&\le \int^{r_*}_0 \left( \frac{2 \beta^2}{n-2} + h^{-2} | \nb_{\td g} h |^2_{\td g} +  \rho_0 h^{-1} \right) (\psi-\beta)^q_+ \sqrt{det \td g} d\td r.
\eal
\ee Here $r_*$ is the length of the geodesic from $p$ to the cut-locus, which also depends on the angular part of the geodesic regarded as the image of a line segment under the exponential map. In the above, $L_2$ is the worst term, which will be estimated first.

Notice that $(\psi-\beta)_+$ is Lipschitz in $\td r$. Hence
\[
L_2= \int^{r_*}_0 \pa_{\td r} (\psi - \beta)_+ (\psi-\beta)^q_+ \sqrt{det \td g} d\td r
= \frac{1}{q+1} \int^{r_*}_0 \pa_{\td r}  (\psi-\beta)^{q+1} _+ \sqrt{det \td g} d\td r.
\]After integration by parts, noticing that $\sqrt{det \td g}=O(\td r^{n-1})$ when $\td r$ is small, we deduce:
\[
\al
L_2 &= \frac{1}{q+1} (\psi-\beta)^{q+1} _+ \sqrt{det \td g} |_{\td r = r_*} -
\frac{1}{q+1} \int^{r_*}_0   (\psi-\beta)^{q+1} _+ \pa_{\td r} \sqrt{det \td g} d\td r\\
&\ge - \frac{1}{q+1} \int^{r_*}_0   (\psi-\beta)^{q+1} _+ \Delta_{\td g} \td r \sqrt{det \td g} d\td r.
\eal
\]Here we have used the standard formula outside of the cut-locus:
\[
\pa_{\td r}  \ln \sqrt{det \td g} = \Delta_{\td g} \td r.
\]Noticing that
\[
\Delta_{\td g} \td r = \psi + \frac{n-1}{\td r},
\]we find a lower bound for $L_2$:
\be
\lab{l2xiajie}
L_2 \ge - \frac{1}{q+1} \int^{r_*}_0   (\psi-\beta)^{q+1} _+  (\psi + \frac{n-1}{\td r}) \sqrt{det \td g} d\td r
\ee Substituting this to \eqref{psint1}, we find
\be
\lab{psint2}
\al
& \frac{1}{n-1} \int^{r_*}_0 \psi^2 (\psi-\beta)^q_+ \sqrt{det \td g} d\td r - \frac{1}{q+1} \int^{r_*}_0   (\psi-\beta)^{q+1} _+  \psi  \sqrt{det \td g} d\td r \\
&\qquad + \int^{r_*}_0 \left(\frac{2 \a }{n-1}  + \frac{2 \beta}{n-2} \right) \psi  (\psi-\beta)^q_+ \sqrt{det \td g} d\td r \\
&\qquad + \int^{r_*}_0 \left( 2 \a + \frac{2(n-1) \beta}{n-2} + 2 \beta \right) \frac{1}{\td r}  (\psi-\beta)^q_+ \sqrt{det \td g} d\td r \\
&\le \int^{r_*}_0 \left( \frac{2 \beta^2}{n-2} + h^{-2} | \nb_{\td g} h |^2_{\td g} +  \rho_0 h^{-1} \right) (\psi-\beta)^q_+ \sqrt{det \td g} d\td r.
\eal
\ee Here we have used $L_3$ to cancel part of $L_2$, i.e.
\[
\al
&L_3 - \frac{n-1}{q+1} \int^{r_*}_0   (\psi-\beta)^{q+1} _+  \frac{1}{\td r} \sqrt{det \td g} d\td r \\
&=\int^{r_*}_0 \frac{2 \psi}{\td r} (\psi-\beta)^q_+ \sqrt{det \td g} d\td r -\frac{n-1}{q+1} \int^{r_*}_0   (\psi-\beta)^{q+1} _+  \frac{1}{\td r} \sqrt{det \td g} d\td r \\
&=\int^{r_*}_0 \frac{2 (\psi-\beta)}{\td r} (\psi-\beta)^q_+ \sqrt{det \td g} d\td r -\frac{n-1}{q+1} \int^{r_*}_0   (\psi-\beta)^{q+1} _+  \frac{1}{\td r} \sqrt{det \td g}
 d\td r \\
 &\qquad +\int^{r_*}_0 \frac{2 \beta}{\td r} (\psi-\beta)^q_+ \sqrt{det \td g} d\td r \\
&=\left(2-\frac{n-1}{q+1} \right) \int^{r_*}_0  \int^{r_*}_0   (\psi-\beta)^{q+1} _+  \frac{1}{\td r} \sqrt{det \td g}
 d\td r
 +\int^{r_*}_0 \frac{2 \beta}{\td r} (\psi-\beta)^q_+ \sqrt{det \td g} d\td r \\
&
\ge \int^{r_*}_0 \frac{2 \beta}{\td r} (\psi-\beta)^q_+ \sqrt{det \td g} d\td r, \qquad \text{due to} \quad q>n-2.
\eal
\]The last term was added to the third line of \eqref{psint2}. Note the power on $\psi-\beta$ in the last term is one order lower, which is important for us later.

Recall that in the geodesic polar coordinates $\sqrt{det \td g}=\sqrt{det \td g}(\td r, \theta)$ with $\theta \in \mathbb{S}^{n-1}$. Integrating \eqref{psint2} on $\mathbb{S}^{n-1}$, we infer, since the cut-locus has measure $0$, that
\be
\lab{psint3}
\al
& \frac{1}{n-1} \int_M \psi^2 (\psi-\beta)^q_+  d\td g - \frac{1}{q+1} \int_M   \psi (\psi-\beta)^{q+1} _+   d\td g \\
&\qquad + \int_M \left(\frac{2 \a }{n-1}  + \frac{2 \beta}{n-2} \right) \psi  (\psi-\beta)^q_+  d\td g  \\
&\qquad + \int_M \left( 2 \a + \frac{2(n-1) \beta}{n-2} + 2 \beta \right) \frac{1}{\td r}  (\psi-\beta)^q_+  d\td g  \\
&\le \int_M \left( \frac{2 \beta^2}{n-2} + h^{-2} | \nb_{\td g} h |^2_{\td g} +  \rho_0 h^{-1} \right) (\psi-\beta)^q_+  d\td g.
\eal
\ee Note that $M$ can be replaced by any sectorial domains $B_\Lambda(p_0, r)$ originating from $p_0$. Here
\be
\lab{defbsector}
B_\Lambda(p_0, r)= \{ exp_{p_0}( t \theta) \, | \, 0 \le t <r, \, \theta \in \Lambda \subset \mathbb{S}^{n-1}  \}.
\ee Note $\psi=\psi-\beta+\beta$ so that
\[
\al
&\int_M \psi^2 (\psi-\beta)^q_+  d\td g = \int_M (\psi-\beta)^{q+2}_+  d\td g
+ 2 \int_M \beta (\psi-\beta)^{q+1}_+  d\td g + \int_M \beta^2 (\psi-\beta)^q_+  d\td g, \\
&\int_M \psi (\psi-\beta)^{q+1}_+  d\td g = \int_M (\psi-\beta)^{q+2}_+  d\td g
+ \int_M \beta (\psi-\beta)^{q+1}_+  d\td g.
\eal
\]These and \eqref{psint3} imply that
\be
\lab{psint3.1}
\al
& \left(\frac{1}{n-1} - \frac{1}{q+1} \right) \int_M  (\psi-\beta)^{q+2}_+  d\td g \\
&\le \int_M \underbrace{\left(\frac{1}{q+1}- \frac{2 \a +2 \beta }{n-1}  - \frac{2 \beta}{n-2} \right)}_{f_1}   (\psi-\beta)^{q+1}_+  d\td g  \\
&\qquad - \int_M \underbrace{\left( 2 \a + \frac{2(n-1) \beta}{n-2} + 2 \beta \right)}_{f_2} \frac{1}{\td r}  (\psi-\beta)^q_+  d\td g  \\
&\qquad + \int_M \underbrace{\left[ \left( \frac{2 \beta^2}{n-2} + h^{-2} | \nb_{\td g} h |^2_{\td g} +  \rho_0 h^{-1} \right) - \left(\frac{2 \a }{n-1}  + \frac{2 \beta}{n-2} \right) \beta \right]}_{f_3} (\psi-\beta)^q_+  d\td g\\
&\equiv R_1+ R_2 + R_3.
\eal
\ee Since $q>n-2$, the left hand side is positive. The main point is that the right hand side is dominated by the left hand side so that the $L^{q+2}(M)$ norm of $\psi_+$ is bounded. Let us check this claim term by term.

First, from \eqref{afajie} and \eqref{nabgtdg},
\[
|\a| \le  c  |\nb_{\td g} h |_{\td g} = c h^{-1} |\nb h|.
\]According to Propositions \ref{prphijie} and \ref{prhphi2}, for any $p \ge 10$, we have, since $h=\phi^2$,
\be
\lab{afalpmo}
\Vert \a \Vert_{L^p(M, \td g)} \le c \Vert h^{-1} |\nb h| \Vert_{L^p(M, \td g)}
=c \Vert h^{-1} |\nb h| h^{n/(2p)} \Vert_{L^p(M,  g)} \le \td b_p.
\ee Here  and later in this section $\td b_p$ is a constant that depends only on the main parameters of the initial metric and $p$, whose value may change from line to line. Likewise, from the definition of the function $\beta$ in \eqref{defbeta}, i.e.
\[
\beta = \frac{n-2}{2} h^{-1} \pa_{\td r} h,
\]we also have
\be
\lab{betalpmo}
\Vert \beta \Vert_{L^p(M, \td g)}  \le \td b_p.
\ee Therefore the three  braced functions in \eqref{psint3.1} satisfy
\be
\lab{lpnormf123}
\Vert f_i \Vert_{L^p(M, \td g)}  \le \td b_p, \quad i=1, 2, 3.
\ee  This implies, by H\"older inequality, that
\[
\al
&|R_1| \le \Vert (\psi-\beta)_+ \Vert^{q+1}_{L^{q+2}} \, \Vert f_1 \Vert_{L^{q+2}} \le
\td b_p \Vert (\psi-\beta)_+ \Vert^{q+1}_{L^{q+2}};\\
&|R_3| \le \Vert (\psi-\beta)_+ \Vert^{q}_{L^{q+2}} \, \Vert f_3 \Vert_{L^{(q+2)/2}} \le
\td b_p \Vert (\psi-\beta)_+ \Vert^{q}_{L^{q+2}}.
\eal
\]
\[
\al
|R_2| &\le \left( \int_M (\psi-\beta)^{q+2}_+ d \td g \right)^{q/(q+2)} \,
\left( \int_M \frac{|f_2|^{(q+2)/2}}{\td r^{(q+2)/2}} d \td g \right)^{2/(q+2)}\\
&\le \Vert (\psi-\beta)_+ \Vert^{q}_{L^{q+2}} \, \Vert 1/\td r \Vert_{L^{(q+2)(1+\e)/2}} \, \Vert f_2 \Vert_{L^{(q+2)(1+\e)/(2 \e)}}.
\eal
\]Fixing a sufficiently small $\e>0$, we have
\[
(q+2)(1+\e)/2<n
\]since $q<2n-2$ by assumption. Thus
\[
\Vert 1/\td r \Vert_{L^{(q+2)(1+\e)/2}} \le c_n
\]which is a constant depending only on the volume part in the basic properties, the diameter of $M$, $\e$, and $q$. Therefore
\[
|R_2| \le \td b_p \Vert (\psi-\beta)_+ \Vert^{q}_{L^{q+2}}.
\]Substituting the preceding bounds on $|R_i|$, $i=1, 2, 3$ into
\eqref{psint3.1}, we see that
\[
\int_{\{ \psi \ge 0 \}}   (\psi-\beta)^{q+2}_+  d\td g \le \int_M  (\psi-\beta)^{q+2}_+  d\td g  \le \td b_q.
\]In case $\psi(x) \ge 0$ and $\psi(x) \ge \beta(x)$, we have
\[
\psi^{q+2}(x) = (\psi(x)-\beta(x) + \beta(x))^{q+2} \le 2^{q+2} \left[ (\psi(x)-\beta(x))^{q+2}_+
+|\beta(x)|^{q+2} \right].
\]In case $\psi(x) \ge 0$ but $\psi(x) \le \beta(x)$, then
\[
\psi^{q+2}(x) \le
\beta^{q+2}(x) =|\beta(x)|^{q+2}.
\]These infer
\[
\int_M \psi^{q+2}_+ d \td g \le c_q \int_{\{ \psi \ge 0 \}}   (\psi-\beta)^{q+2}_+  d\td g +
c_q \int_M |\beta(x)|^{q+2} d \td g \le \td b_q.
\] The proof of the theorem is done by renaming $q+2$ as $q$. \qed

{\remark
\lab{rmth1.2ball}
One can also replace $M$ in the integrals starting from \eqref{psint3} by any proper ball $B(x_0, r, \td g)$ and  obtain the local $L^q(B(x_0, r, \td g))$ norm for $\psi_+$ which also depends on $r$.
}

\medskip

\proof (of Corollary \ref{volcomp})

(a). Let us recall the well known inequality for any Riemmann manifold $(M, g)$:
 \be
\label{differential of volume ratio}
\frac{d}{dr}\left(\frac{|B(x,r)|}{r^n}\right)\leq \frac{1}{r^{n+1}}\int_{\mb{S}^{n-1}}\int_0^r s\,\psi_+ \, \sqrt{det g(s,\theta)} dsd\theta,
\ee where $\psi_+ = (\Delta r - \frac{n-1}{r})_+$.  See Lemma 2.3 in \cite{ZqZm:1} for a quick proof e.g.
Applying this to $\td g$ and using the bound for $\psi_+$ in Theorem \ref{thplapc}, we deduce, for $q \in (n, 2n)$, that
\[
\al
\frac{d}{dr}\left(\frac{|B(x,r, \td g)|_{\td g}}{r^n}\right) &\leq \frac{1}{r^n} \int_{B(x, r, \td g)} \psi_+ \, d \td g
\le \frac{1}{r^n} \Vert \psi_+ \Vert_{L^q}
\, |B(x, r, \td g)|^{(q-1)/q} \\
&\le \frac{\td b_q}{r^n}
\, (r^n/\kappa)^{(q-1)/q} \le \frac{\td b_q}{r^{n/q}}.
\eal
\]Here we have used the $\kappa$ non-inflating property in basic property 5 and renamed
the constant $\td b_q$.

Similarly
\[
\frac{|\partial B(x, r_2, \td g)|_{\td g}}{r^{n-1}_2} -
\frac{|\partial B(x, r_1, \td g)|_{\td g}}{r^{n-1}_1}
 \le \int_{B(x, r_2, \td g)\backslash B(x, r_1, \td g)}  \frac{\psi_+}{r^{n-1}} d\td g
 \le \td b_q \left(r^{(q-n)/q}_2 - r^{(q-n)/q}_1 \right).
\]

(b). This follows from part (a) and the $C^{1/2}$ bound for $h$.
Indeed, by Theorem \ref{thplapc} (a), there exists a positive constant $\alpha>1$, depending only on $b_1$ and $b_2$, such that
\[
B(x, r) \subset B(x, \a r, \td g) \subset B(x, \a^2 r), \quad r>0.
\]Given $r_2>0$, write $h_2 = \sup_{B(x, \a^2 r_2)} h$ and $h_1 = inf_{B(x, \a^2 r_2)} h$.
Using the distance formula for geodesics, it is easy to see that, for any $r \in (0, r_2]$, we have
\[
B(x, r) \subset B(x, \sqrt{h_2} \,  r, \td g), \quad B(x, \sqrt{h_1} \,  r, \td g) \subset B(x, r).
\]Using these and part (a) of the corollary, we obtain
for $r_1 \in (0, r_2]$, that
\[
\al
&\frac{|B(x, r_2)|}{r^n_2} \le \frac{|B(x, \sqrt{h_2} r_2, \td g)|}{r^n_2} =
\frac{|B(x, \sqrt{h_2} r_2, \td g)|_{h g/h}}{r^n_2}
\le \frac{|B(x, \sqrt{h_2} r_2, \td g)|_{h g/h_1}}{r^n_2}\\
&=\frac{|B(x, \sqrt{h_2} r_2, \td g)|_{\td g}}{(\sqrt{h_1} r_2)^n}
  = \left(\frac{h_2}{h_1} \right)^{n/2} \frac{|B(x, \sqrt{h_2} r_2, \td g)|_{\td g}}{(\sqrt{h_2} r_2)^n}\\
 &\le \left(\frac{h_2}{h_1} \right)^{n/2} \left[\frac{|B(x, \sqrt{h_1} r_1, \td g)|_{\td g}}{(\sqrt{h_1} r_1)^n} + (1-n/q)^{-1} \td b_q (\sqrt{h_2} r_2)^{1-n/q} \right]\\
 &\le \left(\frac{h_2}{h_1} \right)^{n/2} \left[\frac{|B(x, r_1)|_{h g}}{(\sqrt{h_1} r_1)^n} + (1-n/q)^{-1} \td b_q (\sqrt{h_2} r_2)^{1-n/q} \right]\\
 &\le \left(\frac{h_2}{h_1} \right)^{n/2} \left[\frac{|B(x, r_1)|_{h_2 g}}{(\sqrt{h_1} r_1)^n} + (1-n/q)^{-1} \td b_q (\sqrt{h_2} r_2)^{1-n/q} \right].
\eal
\]Therefore
\[
\al
&\frac{|B(x, r_2)|}{r^n_2} \le \left(\frac{h_2}{h_1} \right)^{n/2} \left[ \left(\frac{h_2}{h_1} \right)^{n/2} \frac{|B(x, r_1)|}{ r_1^n} + (1-n/q)^{-1} \td b_q (\sqrt{h_2} r_2)^{1-n/q} \right].
\eal
\]Since
\[
|(|h_2/h_1)-1| \le c \sqrt{r_2}
\]due to part (b) of the theorem, we conclude, since $q<2n$ and $r_2$ is bounded, that
\[
\frac{|B(x, r_2)|}{r^n_2}  - \frac{|B(x, r_1)|}{ r_1^n}  \le b_q r_2^{1-n/q}.
\]Here the constant $b_q$ is renamed.
\qed

\section{Proof of Theorem \ref{thhtconj}, convergence
 }

 This section is divided into two subsections.

 \subsection{Preliminaries}

Before starting the proper proof, we need to collect or prove a number of properties which hold uniformly for each time slice of the KRF.

We fix a time slice of the KRF: $(M, g(t))$ and denote
\be
\lab{gtgwanh2f}
g=g(t),  \, \td g= h g, \,   h=h(x, t)=e^{2 f(x, t)},
\ee which is the conformal function determined from Theorem \ref{thplapc} for $g(t)$. If no confusion arises, we will omit $t$. All constants in this section depend only on the basic parameters of the initial metric $g(0)$, unless stated otherwise.

The following properties of $(M, \td g)$ will be used later in the section.

(1). $diam (M, \td g) \le C.$

(2). $|B(x, r, \td g)|_{\td g} \ge \kappa r^n$,  $r \in (0, diam (M, \td g))$.

(3). $
|B(x, r, \td g)|_{\td g} \le \kappa^{-1}  r^n
$ for all $r>0$.

(4). There exists a uniform constant $S_2$ so that the following $L^2$ Sobolev inequality holds: \\
 \[
 \left( \int_{M}  v^{2n/(n-2)} d \td g \right)^{(n-2)/n} \le S_2 \left( \int_{M} | \nabla_{\td g} v |^2_{\td g}
 d\td g +  \int_{M} v^2
 d\td g \right)
\]for all $ v \in C^\infty(M, \td g)$.

(5). uniform $L^2$ Poincar\'e inequality: for any $v \in
C^\infty(B(x, r, \td g))$ where $B(x, r, \td g)$ is a proper ball in $({M},
\td g)$, there is a uniform constant $C$ such that
 \be \lab{l2PIwan}
\int_{B(x, r, \td g)} |v-  \td v_B |^2 d\td g \le C r^2 \int_{B(x, r, \td g)} |\nb_{\td g} v|^2 d\td g.
\ee Here  $\td v_B$
is the average of $v$ in $B(x, r, \td g)$ under $\td g$.

Properties (1)-(4) are direct consequences of the corresponding basic properties for $(M, g(t))$ and the uniform $C^{1/2}$ bound for the conformal factor $h$ since $\td g = h g(t)$. Property (5) can be obtained in the following way. From (1)-(4), it is standard to derive the Gaussian upper bound for the heat kernel on $(M, \td g)$. Then, using the gradient bound in Proposition \ref{pr3.2}, one can derive a Gaussian lower bound for the heat kernel. Namely,
  let $p=p(x, s, y)$ be the (stationary) heat kernel for
$({M}, \td g(t))$.  There exist positive constants $a_1$ and $a_2$,
depending only on the basic parameters of $g(0)$ such that

\be
\lab{gulbwan}
 \frac{a_1}{s^{n/2}} e^{-a_2
d(x, y, \td g)^2/s} \le p(x, s, y) \le \frac{1}{a_1 s^{n/2}} e^{- d(x,
y, \td g)^2/(a_2 s)}, \qquad s \in (0, 1]; \ee

 A combination of the bounds yield a Harnack inequality and the Poincar\'e inequality in (5), c.f. Sec. 6 \cite{Gri} e.g. Alternatively, one can use the segment inequality below.

(6). Weak Harnack inequality for Poisson equations

Let $f$ be a nonnegative solution to the equation on $(M, \td g)$:
\[
\Delta_{\td g} f \le \xi, \quad \xi \in L^q(M, \td g),  \, q>n/2.
\]Then, for $r \in (0, diam(M, \td g)/2]$, $x \in M$,
\be
\lab{wharnakpe}
\oint_{B(x, r, \td g)} f d \td g \le C (f(x) + r^{2-(n/q)} \, \Vert \xi \Vert_{L^q(B(x, r, \td g))}).
\ee The proof is the same as Corollary 2.16 in \cite{TZz:1} using \eqref{gulbwan} and (2).

(7). Ricci curvature for $\td g$.

Recall the formula, c.f. \cite{Be:1} p59, e.g.,
\be
\lab{ricfwan}
Ric_{\td g}=Ric - (n-2) Hess \, f + (n-2) df \otimes df - \Delta f \, g - (n-2) |\nb f|^2 g.
\ee Note the sign of $\Delta f$ is minus while the plus sign in that book was a typo. By  \eqref{eqffphi}, the equation for $f$,
\[
\Delta f = - |\nb f|^2 - V + \rho_0.
\]Hence
\be
\lab{ricfwan2}
Ric_{\td g}=Ric - (n-2) Hess \, f + (n-2) df \otimes df +V \, g - (n-3) |\nb f|^2 g- \rho_0 g.
\ee By our choice that $V=|Ric^-|$, this implies a lower bound for $Ric_{\td g}$:
\be
\lab{ricfwan3}
Ric_{\td g} \ge  - (n-2) Hess \, f - (n-3) |\nb f|^2 g - \rho_0 g.
\ee By \eqref{hesschange}
\be
Hess \, f =  Hess_{\td g} \, f + A
\ee where $A$ is the tensor given by
\be
A_{ij}=\frac{1}{2} f^{-1} g^{kl} (\pa_i f g_{jl}+ \pa_j f g_{li}- \pa_l f g_{ij}) \pa_k h.
\ee Therefore, for some uniform constant $c>0$, we have
\be
\lab{ricfwan4}
Ric_{\td g} \ge  - (n-2) Hess_{\td g} \, f -c |\nb_{\td g} h|^2_{\td g} \,  \td g- c \rho_0 \td g.
\ee

(8). additional a priori bounds for $Ric_{\td g}$ and $Rm_{\td g}$.

\begin{proposition}
\lab{pr3.00}
Let $(M, \td g)$ be given by \eqref{gtgwanh2f} at the beginning of the section.
There exists a positive constant $A_0$ depending only on the basic parameters of $g(0)$ such that
\be
\lab{ricwanma}
\int_{B(x, r, \td g)} |\Ric_{\td g}|^2_{\td g} d\td g \le A_0 r^{n-2}, \quad r \in (0, 1]; \quad \int_M \frac{|Ric_{\td g}(y)|^2}{d_{\td g}(x, y)^{n-2}} d\td g(y) \le A_0;
\ee
\be
\lab{rml2wan}
\int_M |Rm_{\td g}|^2_{\td g} d \td g \le A_0 \int_M |Rm|^2 d g + A_0.
\ee
\end{proposition}
\proof

According to \eqref{ricfwan2}, we have,
\be
\lab{dfdf2-v}
\al
&\Delta f = - |\nb f|^2 - V + \rho_0 = 0,\\
&Ric_{\td g}=Ric - (n-2) Hess \, f + (n-2) df \otimes df +V \, g - (n-3) |\nb f|^2 g- \rho_0 g,
\eal
\ee where $V=|Ric^-|$. Again, during the proof, geometric quantities without index or reference are with respect to the original metric $g$.

From Lemma 2.1 in \cite{TZq2} (or \eqref{defkato} here), we know, since $u$ there is the Ricci potential, that
\be
\lab{tzq22.1}
\int_{B(x, r)} |\Ric|^2 d g \le A_0 r^{n-2}
\ee Pick a standard cut-off function $\eta$ supported in $B(x, 2 r)$ such that $\eta=1$ in $B(x, r)$ and $|\nb \eta| \le c/r$. Then standard computation using Bochner's formula and integration by parts imply that
\[
\al
&\int_{B(x, 2 r)} |\Hess f |^2 \eta^2 d g \\
&\le C \int_{B(x, 2 r)} |\Delta f |^2 \eta d g  + C \int_{B(x, 2 r)} |\nb f |^2 \,
|\nb \eta |^2 d g + C \int_{B(x, 2 r)} |Ric| \, |\nb f |^2 \,
\eta^2 d g\\
&\le C \int_{B(x, 2 r)} |\, |\nb f|^2 + |Ric^-| + \rho_0 |^2 \eta d g  + \frac{C}{r^2} \int_{B(x, 2 r)} |\nb f |^2 \,
 d g + C \int_{B(x, 2 r)} |Ric| \, |\nb f |^2 \,
\eta^2 d g.
\eal
\]Here we just used \eqref{dfdf2-v}.  Since, for all $q \ge 1$,  $\Vert \nb f \Vert_{L^q(M)} \le C_q$ by Theorem \ref{thplapc}, this infers, by \eqref{tzq22.1}, that
\be
\lab{hessfm2}
\int_{B(x,  r)} |\Hess f |^2 dg \le C_\e r^{n-2-\e}
\ee for any $\e>0$. By Cauchy-Schwarz inequality and volume non-inflation property, this implies
\be
\lab{hessfm1}
\int_{B(x,  r)} |\Hess f | dg \le C_\e r^{n-1-(\e/2)} \le C r^{n-2}.
\ee Substituting this to the second identity in
\eqref{dfdf2-v}, we find
\be
\int_{B(x, r)} |\Ric_{\td g}|^2 dg \le A_0 r^{n-2},
\ee which induces, since $\td g= h g$ that
\be
\int_{B(x, r, \td g)} |\Ric_{\td g}|^2_{\td g} d\td g \le A_0 r^{n-2}.
\ee Here $A_0$ may have changed.

Furthermore, by \eqref{defkato}, and \eqref{hessfm1}, one can derive from
 \eqref{dfdf2-v} directly using dyadic annulus that
\[
\int_M \frac{|Ric_{\td g}(y)|^2}{d_{\td g}(x, y)^{n-2}} d\td g(y) \le A_0.
\]This proves the first inequality in \eqref{ricwanma}.

Next, we recall the formula for curvature tensors under conformal change. Let $\td Rm$ be the curvature tensor of the metric $\td g  = e^{2 f} g$, then by \cite{Be:1} p58 e.g.,
\be
\td R_{ijkl}= e^{2f} \left(R_{ijkl} + g \odot (- \Hess f + d f \otimes df - \frac{1}{2} |\nb f|^2 g ) \right).
\ee Here $\odot$ is the Kulkani-Nomizu product. Using \eqref{hessfm2} and bounds on $f$ and $|\nb f|$ in Theorem \ref{thplapc} (a) and (b), we deduce
\eqref{rml2wan}, completing the proof of the proposition.
\qed

\medskip

Now we are in a position to give

\subsection{ $\e$ splitting map, metric cone property, codimension 4 singular set}

The proof  is divided into 4 steps. We mainly work on the conformal metric, extending the Cheeger-Colding theory to the current case in steps 1, 2, 3, reaching convergence modulo co-dimension 2. In step 4, we generalize the main result in the K\"ahler case in Cheeger-Colding-Tian \cite{CCT} to the current case. Due to the properties of the conformal function $h$ in Theorem \ref{thplapc},  we can go back to the original metric to reach the claimed results.

{\it Step 1.}
Mean value inequalities, gradient estimates for harmonic and caloric functions for the conformal metric $\td g= h g$.

\begin{proposition}
\lab{pr3.1}
Let $u$ be a harmonic function with respect to $\td g$ on $M$, i.e.
\be
\lab{eqharmwan}\Delta_{\td g} u=0.
\ee For any $0<r \le diam (M, \td g)/2$, there exists a constant $C$, depending only on the basic parameters of the initial metric, such that
\be
\lab{mviduwan2.0}
| u(x)|^2 \le \frac{C}{|B(x, r, \td g)|_{\td g}} \int_{B(x, r, \td g)}  u^2(y) d\td g(y),
\ee
\be
\lab{mviduwan2}
|\nb_{\td g} u(x)|^2_{\td g} \le \frac{C}{r^2 |B(x, r, \td g)|_{\td g}} \int_{B(x, r, \td g)}  u^2(y) d\td g(y).
\ee
\end{proposition}
\proof

We will just give a proof of the 2nd inequality since the first one is a routine consequence of the Sobolev inequality d) and volume properties b) and c) via Moser's iteration.

It is convenient to work on the Laplace  for the  original metric $g$ and treat those for $\td g$ as  perturbations.
Recall, from \eqref{lapchange} and \eqref{nabchange}
\be
\lab{lapnabchange}
\Delta_{\td g} u  = h^{-1} \Delta u + \frac{n-2}{2} h^{-2} < \nb h, \nb u>, \quad
|\nabla_{\td g} u |^2_{\td g} = h^{-1} | \nabla u |^2.
\ee Therefore
\be
\lab{dduwan}
\Delta u = - \frac{n-2}{2} h^{-1} < \nb h, \nb u>.
\ee Combing with Bochner's formula for $g$,
\be
\lab{bochn3}
\Delta ( |\nb  u |^2)  =  2 | Hess \, u|^2 + 2 < \nb \Delta u, \nb u> + 2 Ric(\nb u, \nb u),
\ee we deduce
\be
\lab{modcdineq2.1}
\al
 \Delta ( |\nb u |^2)  \ge  2 | Hess \, u|^2 - (n-2) < \nb ( h^{-1} < \nb h, \nb u>), \nb u>   - 2 V | \nb u|^2.
\eal
\ee Given $r>0$ such that $2r \le diam (M)$, denote $F= |\nb u |^2$ and pick a smooth cut-off function $\eta$ such that $\eta=0$ outside of $B(x, 2 r)$. For $p \ge 1$, using
$F^{2p+1} \eta^2$ as a test function on \eqref{modcdineq2.1}, we find
\be
\lab{cacioF2}
\al
\int_M |\nb (F^{p+1} \eta)|^2 dg &\le c \int_M F^{2(p+1)} |\nb \eta|^2 dg \\
&\qquad +
\frac{(p+1)^2}{2p+1} (n-2) \underbrace{ \int_M < \nb ( h^{-1} < \nb h, \nb u>), \nb u>  F^{2p+1} \eta^2 dg}_{T_2}
\\
&\qquad +
 \frac{(p+1)^2}{2p+1}  \int_M  2 V F^{2(p+1)} \eta^2 dg
\eal
\ee After integration by parts, we find
\[
T_2 = -\int_M  h^{-1} < \nb h, \nb u> \,  \Delta u  F^{2p+1} \eta^2 dg
-\int_M  h^{-1} < \nb h, \nb u> \,  <\nb u,  \nb (F^{2p+1} \eta^2)> dg
\]which infers, via \eqref{dduwan}, that
\[
T_2 \le \frac{n-2}{2} \int_M  \left[ h^{-1} < \nb h, \nb u> \right]^2  \,   F^{2p+1} \eta^2 dg
+ \int_M  h^{-1} |\nb h|   \,  F  |\nb (F^{2p+1} \eta^2)|  dg.
\]Substituting this to the right hand side of \eqref{cacioF2} and using Cauchy Schwarz inequality, we obtain
\be
\lab{cacioF2.1}
\al
\int_M |\nb (F^{p+1} \eta)|^2 dg &\le c(1+p^2) \int_M F^{2(p+1)} |\nb \eta|^2 dg \\
&\qquad + c p^2  \int_M  h^{-1} |\nb h|   \,  F^{2p+1}  |\nb F|  \eta^2  dg +
c p  \int_M (h^{-1} |\nb h|)^2  F^{2p+2} \eta^2 dg
\\
&\qquad +
 c p  \int_M   V F^{2(p+1)} \eta^2 dg.
\eal
\ee This implies, by Cauchy-Schwarz inequality again,
\be
\lab{cacioF2.2}
\al
\int_M |\nb (F^{p+1} \eta)|^2 dg &\le c(1+p^2) \int_M F^{2(p+1)} |\nb \eta|^2 dg \\
&\qquad +
c p^2  \int_M (h^{-1} |\nb h|)^2  F^{2p+2} \eta^2 dg
 +
 c p  \int_M   V F^{2(p+1)} \eta^2 dg.
\eal
\ee By Theorem \ref{thplapc} (a) and (b), we know that $h^{-1} |\nb h | \in L^{1+n}(M)$ so that $[h^{-1} |\nb h |]^2$ is in $L^{(1+n)/2}$ with a uniform norm. Since $(1+n)/2>n/2$ (the critical number), by the uniform Sobolev inequality and standard interpolation, we know the following embedding inequality holds: $\forall \e>0$, there exists a fixed power $q_0$ and constant $C$ such that
\[
\int_M (h^{-1} |\nb h|)^2  F^{2p+2} \eta^2 dg \le \e \int_M |\nb (F^{p+1} \eta)|^2 dg
+ C \e^{- q_0} \int_M  F^{2p+2} \eta^2 dg.
\]This and the embedding inequality \eqref{vsubcr} for $V$, together with d),  allows us to carry out Moser's iteration from \eqref{cacioF2.2}. Hence
\be
\lab{mvidu1}
|\nb u(x)|^2 \le \frac{C_1}{|B(x, r)|} \int_{B(x, r)} |\nb u(y)|^2 dg(y).
\ee Using \eqref{lapnabchange} and the bounds on $h$, we can find a constant $C_2$ such that
\be
\lab{mviduwan1}
|\nb_{\td g} u(x)|^2_{\td g} \le \frac{C_2}{|B(x, r, \td g)|_{\td g}} \int_{B(x, r, \td g)} |\nb_{\td g} u(y)|^2 d\td g(y).
\ee This easily yields, by using another suitable test function $u \eta_1$ on $\Delta_{\td g} u =0$ and adjusting the radius, that \eqref{mviduwan2} is true. \qed

Next, we present a similar mean value inequality for the heat equation with a different proof which also work for the Laplace equation.

\begin{proposition}
\lab{pr3.2}
Let $u$ be a caloric function with respect to $\td g$, i.e.
\be
\lab{eqrewan}\Delta_{\td g} u - \pa_s u=0, \quad \text{on} \quad M \times [0, \infty).
\ee For any $r \in [0, diam(M, \td g)/2)$, $x \in M$  and $s>2r^2$, there exists a constant $C$, depending only on the basic parameters of the initial metric, such that
\be
\lab{mviheatwan2.0}
| u(x, s)|^2 \le \frac{C}{ |B(x, r, \td g)|_{\td g}} \int^s_{s-r^2} \int_{B(x, r, \td g)}  u^2(y, l) d\td g(y)dl,
\ee
\be
\lab{mviheatwan2}
|\nb_{\td g} u(x, s)|^2_{\td g} \le \frac{C}{r^2 |B(x, r, \td g)|_{\td g}} \int^s_{s-r^2} \int_{B(x, r, \td g)}  u^2(y, l) d\td g(y)dl.
\ee
\end{proposition}
\proof

Again we will just prove the 2nd inequality.
 Bochner's formula for $(M, \td g)$ and \eqref{ricfwan4} together infer that
\be
\lab{bochheatwan}
\al
( \Delta -\pa_s)& ( |\nb_{\td g} u |^2_{\td g})  \ge  2 | Hess_{\td g} \, u|^2_{\td g} + 2 Ric_{\td g}(\nb_{\td g} u, \nb_{\td g} u)\\
& \ge
- (n-2) \underbrace{Hess_{\td g} \, f (\nb_{\td g} u, \nb_{\td g} u)}_{T_1}
-c \underbrace{|\nb h|^2 |\nb_{\td g} u |^2_{\td g}}_{T_2}- c \rho_0 |\nb_{\td g} u |^2_{\td g}.
\eal
\ee Recall from Theorem \ref{thplapc}, for any $p \ge 1$, the $L^p(M)$ norms of $|\nb_{\td g} f|_{\td g}$ and $|\nb_{\td g} h |_{\td g}$ are uniformly bounded. The
mean value inequality can be proven in the same way as Theorem 4.1,  4.3 and Corollary 4.2 in \cite{ZqZm:1}. Here is a sketch of the proof. Write $F=|\nb_{\td g} u |^2_{\td g}$. For $p>1$ and a suitable cut-off function $\eta$, we use $F^{2p+1} \eta^2$ as a test function on \eqref{bochheatwan}. For the term involving $T_1$, we use integration by parts to break the Hessian. For the term involving $T_2$, we use H\"older's inequality and interpolation. Note that $|\nb h|^2$ in $T_2$, as a potential function, is almost bounded since it is in any $L^p$ space. This will give us a Cacciopoli inequality for $F^{p+1} \eta^2$. Since the Sobolev inequality d) holds, we can use Moser's iteration to infer the desired inequality.
\qed

\medskip

{\it Step 2.} $\e$ splitting (harmonic) map,  segment inequality and G-H convergence except on co-dimension $2$ set.

With the help of the integral Laplace comparison in Theorem \ref{thplapc} and the gradient estimate for harmonic functions, the main relevant results of the Cheeger-Colding convergence theory for manifolds with bounded Ricci curvature can be extended to the current case in a relatively standard manner, at least in the co-dimension 2 singularity case. We will give a description of this process below. Identical proofs will be skipped and referred to relevant references.

Step 2.1. excess estimates.   one can derive the excess estimate \cite{AG} by the standard method in \cite{ChCo2} which is further described in \cite{Ch}. See also \cite{CoNa} first and \cite{ZqZm:1} Theorem 5.3 later for a modified proof.  The main ingredients for the latter are the integral Laplace comparison and the weak Harnack inequality (6).
For completeness, we state the result and provide the proof, mimicking that in \cite{ZqZm:1}.

Fix $q_+$ and $q_-$ in $M$, and define the excess function on $(M, \td g)$ as
\[
 e(y)= d_{\td g}(y,q_+)+ d_{\td g}(y,q_-)- d_{\td g}(q_+,q_-).
\]
For any $x\in M$, if we denote the Buseman functions by
\[
 b_{\pm}(y)=d_{\td g}(y,q_{\pm})-d_{\td g}(x,q_{\pm}).
\]

\begin{proposition}[Excess estimate]
\label{excess}
 Let $x \in M$ and $\Lambda \ge 4$ be a constant. Assume that \be\label{excess assumption 1}
d_{\td g}(x,q_{\pm})\geq \Lambda^{\frac{1}{2}}
\ee
and
\be\label{excess assumption 2}
e(x)\leq \e.
\ee
Then for $r \in (0, 1]$, we have, $\forall y \in   B(x,r, \td g)$,
\begin{equation}\label{excess estimate}
e(y) \leq C \Psi;
\end{equation}   here,
\[
\Psi=  \left(\e+ r^{2-(n/q)} \, ||\psi_+ + \psi_-||_{L^q(B(x,r, \td g))}
 + r^2 \Lambda^{-\frac{1}{2}} \right)^{1/(n+1)}
, \quad q \in (n, 2n),
\]and
\be\label{eq psi}
\psi_{\pm}=\left(\Delta_{\td g} d_{\td g}(y,q_{\pm})-\frac{n-1}{d_{\td g}(y,q_{\pm})}\right)_+;
\ee  $C>0$  is a constant depending only on the constant in \eqref{wharnakpe} and volume doubling constant,
\end{proposition}

\begin{proof}
 Note that $\Psi$ is a function which goes to $0$ when $r \to 0$ or $\e \to 0$.
 The proof, which differs from the original one in \cite{AG}, is inspired by Remark 2.9 \cite{CoNa}. During the proof, all geometric quantities are with respect to $\td g$, but for simplicity, we will omit the reference to $\td g$. For example, $d(x, y)=d_{\td g}(x, y)$ etc.
Note that
\[
\Delta e(y)=\Delta b_+ + \Delta b_-\leq \frac{n-1}{d(y,q_+)}+\frac{n-1}{d(y,q_-)}+\psi_++\psi_-.
\]
According to \eqref{wharnakpe},  we have
\[
\al
\oint_{B(x,r)}e(y)dg(y)\leq& C\left( e(x) + r^{2-(n/q)} \, ||\psi_+ + \psi_-||_{L^q(B(x,r))} + 4(n-1)r^2 \Lambda^{-\frac{1}{2}}\right)\\
\leq & C\left(\e+ r^{2-(n/q)} \, ||\psi_+ + \psi_-||_{L^q(B(x,r))}
 + r^2 \Lambda^{-\frac{1}{2}} \right)\\
\equiv &C \Psi_1.
\eal
\]

Let us assume $\Psi_1<1/2$ first.
Then, by using the volume doubling property, for any $y\in B(x,(1-\Psi_2)r)$ and $\Psi_2 \in (0, 1)$, we have
\[\al
\oint_{B(y,\Psi_2 r)}e(z)dz\leq&\frac{|B(x,r)|}{|B(y,\Psi_2 r)|}\oint_{B(x,r)}e(z)dz\\
\leq&\frac{|B(y,2r)|}{|B(y,\Psi_2 r)|}C \Psi_1\\
\leq&  C^2 \Psi_1 \Psi_2^{-n}.
\eal\]
Hence, if we choose $\Psi_2= \Psi_1^{\frac{1}{n+1}}$, then the above inequality becomes
\[
\oint_{B(y,\Psi_2 R)}e(z)dz\leq C^2 \Psi_2.
\]
It implies that there exists  $z\in B(y,\Psi_2 R)$ such that
\[
e(z)\leq C^2 \Psi_2,
\]
and  by mean value theorem
\[
e(y)\leq e(z)+|\nb e| d(y,z)\leq C^2 \Psi_2+2 \Psi_2 r \le (C^2+2) \Psi_2.
\]
Since $y$ is an arbitrary point in $B(x,(1-\Psi_2)R)$, we can apply the  mean value theorem one more time for points in $B(x,R)\setminus B(x,(1-\Psi_2)R)$ to obtain
\[
\Psi= C \Psi_2,
\]where we have renamed the constant.

Next, if $\Psi_1 \ge 1/2$, the conclusion obviously holds since $|\nb e| =2 $ a.e., which infers $|e(y)-e(x)| \le 2 r$ so that $e(y) \le \e + 2 r \le C \Psi$.
\end{proof}

What we really need is a version of the excess estimate in small scales, i.e. when $r \to 0$. To this end we carry out a scaling, for $r \in (0, 1]$:
\be
\lab{defhatgb0}
\hat{g} = \lam \tilde{g}, \qquad \hat e = \lam^{1/2} e, \qquad \lam=1/r^2.
\ee The main point of the next proposition is that the excess goes to $0$ when the error $\e$, the scale $\lam^{-1/2}=r$ and $\Lambda^{-1}$ go to $0$.

\begin{proposition}[scaled excess estimate]
\label{excess2}
 Let $x \in M$ and $\Lambda \ge 4$ be a constant. Assume that
  \be\label{excess assumption 2.1}
d_{\hat g}(x,q_{\pm})\geq \Lambda^{\frac{1}{2}}
\ee
and
\be\label{excess assumption 2.2}
\hat e(x)\leq \e.
\ee
Then  we have, $\forall y \in   B(x,1, \hat g)$, $q \in (n, 2n)$, that
\be
\lab{ehatjie}
\hat e(y) \leq C \left( \e +  2 \beta^*_q \lam^{(n-q)/(2q)}
 + \Lambda^{-\frac{1}{2}} \right)^{1/(n+1)}
\ee where $\beta^*_q$ is the constant in part (c) of Theorem \ref{thplapc}, $C>0$  is a constant depending only on the constant in \eqref{wharnakpe} and volume doubling constant.
\end{proposition}

\proof

Since the proof of the proposition only depend on the weak Harnack inequality \eqref{wharnakpe} and the volume doubling property, its conclusion still hold for $(M, \hat g)$. Therefore, we have
\begin{equation}\label{excess estimate2}
\hat e(y) \leq C \hat \Psi;
\end{equation}   here,
\[
\hat \Psi=  \left( \e +  ||\hat \psi_+ + \hat \psi_-||_{L^q(B(x,1, \hat g))}
 + \Lambda^{-\frac{1}{2}} \right)^{1/(n+1)}
, \quad q \in (n, 2n),
\]and
\be\label{eq psi2}
\hat \psi_{\pm}=\left(\Delta_{\hat g} d_{\hat g} (y,q_{\pm})-\frac{n-1}{d_{\hat g}(y,q_{\pm})}\right)_+;
\ee  $C>0$  is a constant depending only on the constant in \eqref{wharnakpe} and volume doubling constant. By scaling and using Theorem \ref{thplapc} (c), we find
\[
\Vert \hat \psi_{\pm} \Vert_{L^q(B(x,1, \hat g))} = \lam^{(n-q)/(2q)}
\Vert  \psi_{\pm} \Vert_{L^q(B(x,1/\sqrt{\lam}, \td g))}
\le \lam^{(n-q)/(2q)} \beta^*_q.
\]After substituting this to \eqref{excess estimate2}, we finish the proof of the proposition.
\qed

\medskip

Step 2.2. $\e$ splitting map.

On a Riemannian manifold $(M, g)$, let us recall

\begin{definition}
\label{def e splitting}
A harmonic map $v=(v^1, v^2,\cdots, v^k):\ B(x, r)\rightarrow \mathbb{R}^k$ is an $\e$-splitting map, if\\
(1) $|\nb v|\leq 1+\e$ in $B(x, r)$;\\
(2) $\displaystyle \oint_{B(x, r)}\left|<\nb v^i, \nb v^j>-\delta_{ij}\right|^2\leq \e^2$, $\forall i,j$;\\
(3) $\displaystyle r^2\oint_{B(x, r)}|\nb^2 v^i|^2\leq \e^2$, $\forall i$.
\end{definition}

Based on the excess estimate, integral Laplace and volume comparison, and gradient bounds for harmonic functions,  the existence of $\e$ splitting map on $(M, \td g)$  essentially follows from \cite{ChCo2} except that the Ricci curvature term need to be treated differently as follows.

Since $f \in C^{1/2}$ and $|\nb f| \in L^p, \, \forall p \ge 1$ due to Theorem \ref{thplapc}, by \eqref{ricfwan4},  as far as the Ricci  curvature is concerned, we have reached the case which is essentially covered by the extended Bakry-\'Emery condition in \cite{ZqZm:1} Sec.2.2. As pointed out at the end of Sec. 6 of that paper, the existence of $\e$ splitting map is valid with the same proof. Note that the extra term $-c |\nb f|^2 g$ in $Ric_{\td g}$ is almost bounded in the sense that it is in any $L^p$ space for $p \ge 1$. Hence it can be easily be dealt with by H\"older's inequality. In fact this has been done for the case when $p= (n+\e)/2$, with any small $\e>0$,  in \cite{TZz:1}.
For completeness, we describe the treatment for the Ricci term here. The constants below depend only on the basic parameters of $g(0)$ unless stated otherwise. Their values may also change from line to line.

On a proper ball $B(x, r, \td g) \subset M$, let $\mathbf{b_{\pm}}$ be the harmonic functions with respect to $\Delta_{\td g}$ in the $\e$ splitting map. c.f. Lemma 9.13 in \cite{Ch}. They are obtained from solving the Dirichlet problem on the ball with $b_{\pm}$ in the previous sub-step as boundary values.

 We need to show that $ \Hess \, \mathbf{b_{\pm}}$  are small in $L^2$ sense in the half ball when the Kato norm $K(|Ric|^2)$ in \eqref{defkato} is small. The latter happens when one blows up the metric by direct computation. For simplicity, we drop the subscript $\pm$ from $\mathbf{b_{\pm}}$.
By Bochner's formula
\be
\lab{bochwan4}
\frac{1}{2} \Delta_{\td g} |\nb_{\td g} \mathbf{b}|^2_{\td g} = | Hess_{\td g} \, \mathbf{b}|^2_{\td g}
+ Ric_{\td g} (\nb_{\td g} \mathbf{b}, \nb_{\td g} \mathbf{b}).
\ee Choose $\eta$ to be a good cut-off function in the sense of Cheeger-Colding in the ball $B(x, 0.75 r, \td g)$ such that $\eta=1$ in $B(x, r/2, \td g)$, $|\nb_{\td g} \eta| \le c/r$ and $|\Delta_{\td g} \eta| \le c/r^2$.
 The existence of $\eta$ can be proven in a couple of ways by now. The original one relies on solving a suitable Poisson equation and its gradient estimate. One can also use the heat equation to smooth out the distance function, relying on the gradient estimate for the heat equation. Both gradient estimates are available to us due to Propositions \ref{pr3.1},  \ref{pr3.2}.
By \eqref{ricfwan4} for the Ricci curvature and \eqref{bochwan4}, we derive
\be
\al
&\int \eta | Hess_{\td g} \, \mathbf{b}|^2_{\td g} d\td g \le
\frac{1}{2} \int \eta \Delta_{\td g} \left(|\nb_{\td g} \mathbf{b}|^2_{\td g}-1 \right) d\td g \\
 &\qquad +  (n-2) \int \eta Hess_{\td g} \, f (\nb_{\td g} \mathbf{b}, \nb_{\td g} \mathbf{b}) d\td g +c \int \eta |\nb_{\td g} h|^2_{\td g}   |\nb_{\td g} \mathbf{b}|^2_{\td g} d\td g +  c \rho_0 \int \eta |\nb_{\td g} \mathbf{b}|^2_{\td g} d\td g.
\eal
\ee After integration by parts, this implies
\[
\al
&\int \eta | Hess_{\td g} \, \mathbf{b}|^2_{\td g} d\td g \le
\frac{1}{2} \int  \left(|\nb_{\td g} \mathbf{b}|^2_{\td g}-1 \right)  \Delta_{\td g} \eta d\td g \\
 &\qquad +  c \int |\nb_{\td g} f|_{\td g}  | Hess_{\td g} \mathbf{b}|_{\td g} \, |\nb_{\td g} \mathbf{b}|_{\td g}  \eta d\td g
 +c \int |\nb_{\td g} f|_{\td g} \,  |\nb_{\td g} \mathbf{b}|^2_{\td g}  |\nb_{\td g} \eta|_{\td g} d\td g \\
 & \qquad +c \int \eta |\nb_{\td g} h|^2_{\td g}   |\nb_{\td g} \mathbf{b}|^2_{\td g} d\td g +  c \rho_0 \int \eta |\nb_{\td g} \mathbf{b}|^2_{\td g} d\td g.
\eal
\] Using Cauchy-Schwarz inequality, we deduce
\be
\al
\int \eta | Hess_{\td g} \, \mathbf{b}|^2_{\td g} d\td g &\le
 \int  \left(|\nb_{\td g} \mathbf{b}|^2_{\td g}-1 \right)  \Delta_{\td g} \eta d\td g \\
 &\qquad +  c \int |\nb_{\td g} f|^2_{\td g}   \, |\nb_{\td g} \mathbf{b}|^2_{\td g}  \eta d\td g
 +c \int |\nb_{\td g} f|_{\td g} \,  |\nb_{\td g} \mathbf{b}|^2_{\td g}  |\nb_{\td g} \eta|_{\td g} d\td g \\
 & \qquad  +c \int \eta |\nb_{\td g} h|^2_{\td g}   |\nb_{\td g} \mathbf{b}|^2_{\td g} d\td g +  c \rho_0 \int \eta |\nb_{\td g} \mathbf{b}|^2_{\td g} d\td g.
\eal
\ee It is well known from Step 2.1 that $\oint_{B(x, r, \td g)} ||\nb_{\td g} \mathbf{b}|^2_{\td g}-1| d\td g= o(1)$ c.f. Lemma 9.10 \cite{Ch} (as long as one has the integral Laplace comparison result). Also, by the gradient estimate in Proposition \ref{pr3.1}, we have $|\nb_{\td g} \mathbf{b}|_{\td g} \le c$ in
$B(x, 0.75 r, \td g)$.
Hence
\be
\al
&\oint_{B(x, r/2, \td g)}  | Hess_{\td g} \, \mathbf{b}|^2_{\td g} d\td g \le \frac{o(1)}{r^2}
+  c \oint_{B(x, r, \td g)} |\nb_{\td g} f|^2_{\td g}    d\td g
 +\frac{c}{r} \oint_{B(x, r, \td g)} |\nb_{\td g} f|_{\td g} \,  d\td g   +  c \rho_0.
\eal
\ee Here we have used  $h=e^{2 f} = \phi^2$ from \eqref{defh=} and the volume doubling property.

Since $(M, \td g)$ is compact and we are only concerned with its properties in micro scales, we take $r \le 1$ and perform a scaling
\be
\lab{defhatgb}
\hat{g} = \lam \tilde{g}, \qquad \mathbf{\hat b} = \lam^{1/2} \mathbf{b}, \qquad \lam=1/r^2.
\ee Then
\be
\al
&\oint_{B(x, 1/2, \hat g)}  | Hess_{\hat g} \, \mathbf{ \hat b}|^2_{\hat g} d\hat g \le o(1)
+  c \oint_{B(x, 1, \hat g)} |\nb_{\hat g} f|^2_{\hat g}    d\td g
 +c \oint_{B(x, r, \hat g)} |\nb_{\hat g} f|_{\hat g} \,  d\hat g   +  c \rho_0 \lam^{-1}.
\eal
\ee  According to Theorem \ref{thplapc}, there is a uniform constant $c$ such that
\[
\int_M |\nb_{\td g} f |^{n+2}_{\td g} d \td g \le c.
\]Here we just used the fact that $b_1 g \le \td g \le b_2 g$ where $b_1, b_2$ are uniform constants, from the same theorem. Hence
\[
\int_M |\nb_{\hat g} f |^{n+2}_{\hat g}  d\hat g \le c \lam^{-1} \to 0, \quad \lam \to \infty.
\]Using the $\kappa$ non-collapsing property (item (2) in this section),  we deduce the crucial smallness for the Hessian
\be
\lab{hessxiao}
\al
\int_{B(x, 1/2, \hat g)}  | Hess_{\hat g} \, \mathbf{ \hat b}|^2_{\hat g} d\hat g \le o(1), \qquad \lam \to \infty.
\eal
\ee

To summarize, we have obtained

\begin{proposition}
\lab{presplitgwan}
Given any  $\e \in (0, 1]$, there exists $\delta, r_0>0$ such that the following statement is true.  For any $r \in (0, r_0]$,  if $d_{GH}(B(x, r, \td g),  \, B_{\mathbb{R}^n}(0, r))< \delta r$, then there exists an $\e$-splitting map $\mathbf{b}$ from $B(x, r, \td g)$ to
$\mathbb{R}^n$.
\end{proposition}

\medskip

Step 2.3. segment inequality.  One can obtain a version of segment inequality using exactly the argument in \cite{TZz:1} Sec. 2.4, which is sufficient for our purpose. Comparing with the standard segment inequality by Cheeger-Colding, an additional assumption on the boundedness of the integrand is made.  Even better, with just the integral Laplace comparison result in Theorem \ref{thplapc}, a stronger segment inequality can be proven by the method in a more recent paper \cite{Chln} Theorem 1.1. Here it is only assumed that the line integrals of the integrand is bounded. We mention that the main assumption in that paper is that the Ricci curvature is in certain integral space. However, what is really needed is the integral Laplace comparison, as explained below.

The following is a rephrase of a special case of Theorem 1.1 \cite{Chln}. The proof is given in Sec. 3.1 of that paper, with the parameter $H$ there taken as $0$.

\begin{proposition} (segment inequality)
\lab{prsegineq} Let $B(x, 2r) \subset M$ which is a complete $n$ manifold and $u: M \to \mathbb{R} $ be a continuous function. Consider, for $y, z \in M$ the line integrals
\[
F_u(y, z) = \inf \{ \int_{\gamma} u \, | \, \gamma  \, \text{is a minimal geodesic connecting} \, y, z \}.
\]Suppose $|F_u(y, z)| \le \mathbb{C}_0$ for each $(y, z) \in B(x, 2 r) \times B(x, 2r)$.
Then for any two measurable sets $A_1, A_2 \subset B(x, r)$, we have
\[
\int_{A_1 \times A_2} |F_u(y, z)| dg(y) dg(z) \le 2^{n+1} r  (|A_1| +|A_2|) \left( \int_{B(x, 2r)} |u| dg + \mathbb{C}_0  \sup_{y \in B(x, r)} \int_{B(x,  2r)} \psi_+ dg \right)
\]where $\psi_+(\cdot) = (\Delta d(\cdot, y) - \frac{n-1}{d(\cdot, y)})_+$.
\end{proposition}

For us, $g$ in the segment inequality is the $\td g= h g$. Due to Theorem \ref{thplapc} here, the last integral in the segment inequality becomes small if one blows up the metric, which is crucial for us.
\medskip

Step 2.4. With the $\e$ splitting map and segment inequality, volume convergence and almost splitting properties for $(M, \td g)$  can be obtained in the same way as \cite{Co} or \cite{Ch} Sec. 9, Theorem 9.31. See also Theorem 1.3 in \cite{Chln}.

Cone rigidity and
metric cone properties essentially follows from  \cite{ChCo1}, \cite{ChCo2} except that the Ricci curvature term need to be treated differently as in Step 2.2. See Section 6 in \cite{ZqZm:1}, especially Theorems 6.3, 6.5 and 6.6 and the comment after Theorem 6.7, and also Theorem 1.3 in \cite{Chln}  e.g.

For completeness we indicate  the necessary changes in proof of  the key cone rigidity property modeled on Theorem 6.3 of \cite{ZqZm:1}.

\begin{proposition}[Cone rigidity]
\label{pr cone rigidity}

Let $(M, \td g)$ be given in \eqref{gtgwanh2f}.
For any $\delta \in (0, 1]$,  there exists a uniform $r_0>0$ depending  only on $\delta$ and the basic parameters of $g(0)$ such that the following statement is true.  If for $r \in(0, r_0]$,
\be\label{volume condition}
(1-\delta)  |B(x,r, \td g)|_{\td g} \leq \frac{r}{n}  |\pa B(x,r, \td g)|_{\td g},
\ee
then for some $\psi_*=\psi_*(\delta)$ which goes to $0$ as $\delta$  goes to $0$, we have
\be
\label{GH close}
d_{GH}(B(x, r, \td g ), C_{0,r}(Z)) \leq \psi_* r,
\ee
where $Z$ is a length space with
\be\label{diam Z}
diam(Z)\leq \pi+\psi_*,
\ee and $C_{0,r}(Z)$ is a cone over $Z$.
\end{proposition}
\proof

By scaling, it is enough to consider the scaled metric $\hat g = r^{-2} \td g$ and set $r=1$. For simplicity of presentation, we will drop the hat from $\hat g$ in the rest of the proof. So $g$ actually means $\hat g$. Also, if no confusions arise will omit the reference to the metric in all geometric quantities, unless stated otherwise.
In addition all the constants depend only on the basic parameters of $g(0)$

According to \cite{ChCo1}, the key step is to show that the metric is close to the Hessian of the function below in $L^2$ sense under the condition of the proposition.

Given $x\in M$, let $F$ be the smooth solution of the following Dirichlet problem:
\[
\Delta F=1\ in\ B(x,1);\ F |_{\pa B(x,1)}=\frac{1}{2n}.
\]
By using the Green's function $\Gamma_{B(x,1)}$, we can write $F$ as
\be\label{mean value f}
F(z)=\frac{1}{2n}-\int_{B(x,1)}\Gamma_{B(x,1)}(z,y)dg(y).
\ee
By the upper bound of the heat kernel, one obtains a lower bound of $F$, as in Theorem 6.3 of \cite{ZqZm:1}
\be
\label{f lower bound}
\al
F(z)\geq -c_1
\eal
\ee which is a uniform constant. Note that the heat kernel bound persists under blowing up of metrics.

For an upper bound of $F$, let $r_x=d(z, x)$ be the distance function from $x$. Then
\be\label{Delta f-U}
\al
\Delta F-\Delta\frac{r_x^2}{2n}= 1-\frac{r_x}{n}\Delta r_x-\frac{1}{n}
\geq -\frac{r_x}{n}\Psi,
\eal
\ee
where $\Psi=\left(\Delta r_x-\frac{n-1}{r_x}\right)_+$.

According to the inhomogeneous maximum principle for $F-\frac{r_x^2}{2n}$, which can be proven as
Proposition
\ref{pr3.1}, we get
\be\label{f upper bound}
\al
F(z)\leq & \frac{r_x^2}{2n}+C \Vert  \Psi \Vert_{L^q(B(x,1))}, \qquad q >n /2.
\eal
\ee

Next, we show that $F$ and $\frac{r_x^2}{2n}$ are close in integral sense. By the volume condition \eqref{volume condition} in the proposition,
\be
\label{integral upper bound Delta f-U}
\al
\int_{B(x,1)} \left(\Delta F-\Delta \frac{r_x^2}{2n} \right) =&|B(x,1)|-\int_{\partial B(x,1)}
<\nb \frac{r_x^2}{2n}, \nb r_x> dS\\
=& |B(x,1)|-\frac{1}{n} |\pa B(x,1)|\\
\leq & \delta |B(x,1)|.
\eal
\ee
From \eqref{Delta f-U}, we can also deduce the lower bound
\be\label{integral lower bound Delta f-U}
\al
\oint_{B(x,1)} \left(\Delta F-\Delta \frac{r_x^2}{2n} \right) \geq& -\oint_{B(x,1)}
-\frac{r_x}{n}\Psi \ge -C \Vert r_x \Psi \Vert_{L^q(B(x,1))}
\eal
\ee
Therefore, we have
\[
-C \Vert r_x \Psi \Vert_{L^q(B(x,1))}
\leq \oint_{B(x,1)} \left( \Delta F-\Delta \frac{r_x^2}{2n} \right)
\leq \delta.
\]

Also, in \eqref{f upper bound}, setting
$A \equiv C \Vert  \Psi \Vert_{L^q(B(x,1))},
$ we find
\be
\lab{r2x-fa}
\frac{r_x^2}{2n}-f+A \geq0.
\ee
Combining the inequality above and \eqref{f lower bound},  \eqref{Delta f-U}, \eqref{integral upper bound Delta f-U}, we have
\[
\al
0\leq&\oint_{B(x,1)}\left[\Delta (F-\frac{r_x^2}{2n})+\frac{r_x}{n}\psi\right]\left[\frac{r_x^2}{2n}-F+A\right]\\
= &\oint_{B(x,1)} A\Delta (F-\frac{r_x^2}{2n})+(\frac{r_x^2}{2n}-F)\Delta (F-\frac{r_x^2}{2n})+\frac{r_x}{n}\psi\left[\frac{r_x^2}{2n}-F+A\right]\\
\leq& \delta A-\oint_{B(x,1)} \left( |\nb (F-\frac{r_x^2}{2n})|^2+ C [A+c_1] \Psi \right).
\eal
\]
It follows that
\be\label{W2 bound U-f}
\al
\oint_{B(x,1)} |\nb (F-\frac{r_x^2}{2n})|^2\leq& \delta A+ C [A+c_1] A.
\eal
\ee
Therefore, by the Poincar\'e inequality, we have
\be\label{L2 bound U-f}
\oint_{B(x,1)}\left(\frac{r_x^2}{2n}-F \right)^2 \leq c \delta A+ c C [A+c_1] A \equiv \psi_1.
\ee
From \eqref{L2 bound U-f}, volume comparison theorem, mean value inequality, and the bounds of $F$ (\eqref{f lower bound} and \eqref{f upper bound}), it is not hard to see that
\be\label{bound f-U}
\left|\frac{r_x^2}{2n}-F\right|\leq \psi_2, \ on\ B(x,(1-\psi_2)),
\ee
where $\psi_2=\psi^{1/(n+1)}_1$.

Next, we show the closeness of $\Hess F$ and the metric $g$ in the average sense in a slightly smaller ball. Note the combination of  \eqref{bound f-U} and \eqref{W2 bound U-f} gives
\be\label{L1 bound df-dU}
\oint_{B(x,(1-\Psi_2))}\left||\nb F|^2-\frac{2F}{n}\right|\leq \psi_2.
\ee The reason is that they imply
\[
\al
&\oint_{B(x,(1-\Psi_2))}\left| |\nb F|^2-\frac{2F}{n}\right |
\le \oint_{B(x,(1-\Psi_2))}\left| |\nb F|^2-\frac{r^2_x}{n^2}\right | + \psi_2\\
&\le \left(\oint_{B(x,(1-\Psi_2))}\left| |\nb F|+ \frac{r_x}{n}\right |^2 \right)^{1/2}
\left(\oint_{B(x,(1-\Psi_2))}\left| |\nb F|- \frac{r_x}{n}\right |^2 \right)^{1/2} + \psi_2\\
&\le C \left(\oint_{B(x,(1-\Psi_2))}\left| |\nb (F- \frac{r^2_x}{2 n})\right |^2 \right)^{1/2} + \psi_2 \le C \sqrt{\psi_1} + \psi_2 \le \psi_2.
\eal
\]Here we have adjusted $\psi_2$ by a multiplicative constant.

Pick, for now, a function
\be
\lab{psi32}
\psi_3=\psi_3(\delta) \ge  \psi^{1/3}_2,
\ee whose values will be specified later. Let $\phi$ be the good cut-off function in $B(x, 1-0.5 \psi_3)$ such that $\phi\big|_{B(x,(1-\psi_3))}=1$ and $|\Delta \phi| \le C/\psi^2_3$. From \eqref{L1 bound df-dU} and Bochner's formula, we have
\be\label{Hess f 1}
\al
C \psi_3 \geq& \frac{1}{2}\oint_{B(x,1)}\Delta \phi\cdot(|\nb F|^2-\frac{2F}{n})\\
\geq&\frac{1}{2}\oint_{B(x,1)}\phi\cdot\Delta(|\nb F|^2-\frac{2F}{n})\\
=&\oint_{B(x,1)}\phi\left(2 |\nb^2 F|^2+ 2 Ric(\nb F,\nb F)-\frac{2}{n}\right).
\eal
\ee

That is
\[
\al
\oint_{B(x,1)} &\phi\left[|\nb^2 F|^2-\frac{1}{n}\right]
\leq C \psi_3 - \oint_{B(x,1)} \phi  Ric(\nb F,\nb F)\\
&\le C \psi_3 +   \sup_{B(x, 1-0.5 \psi_3)} |\nb_{\hat g} F|^2_{\hat g} \oint_{B(x,1, \hat g)}   |Ric^-_{\hat g}|_{\hat g} d\hat g\\
&=C \psi_3 +   \sup_{B(x, 1-0.5 \psi_3)} |\nb_{\hat g} F|^2_{\hat g} r^2 \oint_{B(x,r, \td g)}   |Ric^-_{\td g}|_{\td g} d \td g \\
&\le
C \psi_3 +   \frac{C  r^2}{\psi^2_3} \oint_{B(x,r, \td g)}   |Ric^-_{\td g}|_{\td g} d \td g
\eal
\]
In the above we have used the interior gradient estimate for $F$, which can be obtained in the same way as Proposition \ref{pr3.1} since the right hand side of the equation for $F$ is the constant $1$.
Recalling the lower bound on the Ricci curvature
\eqref{ricfwan4}
\[
Ric_{\td g} \ge  - (n-2) Hess_{\td g} \, f -c |\nb_{\td g} h|^2_{\td g} \,  \td g- c \rho_0 \td g,
\]we find that
\[
\al
&\oint_{B(x,1)}\phi\left[|\nb^2 F|^2-\frac{1}{n}\right]\\
&\leq C \psi_3 + \frac{C  r^2}{\psi^2_3} \oint_{B(x,r, \td g)}  \left(|Hess_{\td g} \, f |_{\td g} + |\nb_{\td g} h|^2_{\td g} +  \rho_0 \right)  d \td g.
\eal
\] Using  \eqref{hessfm1}, which is an a priori bound, together with the $L^q$ bound on $|\nb f|$ in Theorem \ref{thplapc}, we deduce
\[
\oint_{B(x,1)}\phi\left[|\nb^2 F|^2-\frac{1}{n}\right] \le C \psi_3 + \frac{C r^{0.5}}{\psi^2_3} \le 2 C \psi_3.
\]Here we have chosen $r \le r_0$ with $r_0 \le \psi^6_3$.
Now we have to show that $\psi_3=\psi_3(\delta)$ can be chosen so that it goes to $0$ as $\delta \to 0$. To do this, let us recall that from \eqref{r2x-fa} that
\[
A \equiv C \Vert  \Psi \Vert_{L^q(B(x,1))} \le C r^{1-(n/q)}, \, q \in (n, 2n),
\]where we used Theorem \ref{thplapc} and scaling again. Also recall from
\eqref{L2 bound U-f} that
\[
\psi_1= c \delta A+ c C [A+c_1] A
\]and from \eqref{bound f-U},
\[
\psi_2 = \psi^{1/(n+1)}_1 \le c [\delta r^{1-(n/q)} + c r^{1-(n/q)}]^{1/(n+1)}.
\]From \eqref{psi32}, we just need $\psi_3 \ge \psi^{1/3}_2$. Hence we can always ensure that $\psi_3 \to 0$ as $\delta \to 0$ by choosing $r_0$ sufficiently small one more time,
for example $\psi_3=1/|ln \delta|$, $r_0=\delta$.

Thus, after adjusting the constant, we deduce,
\be\label{bound Hess f}
\al
&\oint_{B(x,(1-\Psi_3))}\left|\nb^2F-\frac{1}{n}g_{ij}\right|^2=
\oint_{B(x,(1-\Psi_3))}\left|\nb^2F-\frac{\Delta F}{n}g_{ij}\right|^2\\
&\leq\oint_{B(x,1)}\phi\left[|\nb^2 F|^2-\frac{1}{n}\right]\leq \psi_3.
\eal
\ee
The rest of the proof is standard, using the segment inequality in step 2.3.

\qed

\medskip

Step 2.5. Given a sequence $t_i \to \infty$, let $\td g_i = h_i(\cdot) g(t_i)$ where $h_i$ is the conformal factor determined from Theorem \ref{thplapc} for $g(t_i)$. With Step 2.4 completed, co-dimension 2 property for singular sets of the G-H limit spaces of a subsequences of $\{ \td g_i \}$ follow directly from the arguments in  \cite{ChCo2}. See also Sec. 10 in \cite{Ch}.
 Recall that the conformal function $h_i(\cdot)$ satisfies Theorem \ref{thplapc} (a) and (b). Therefore,  the co-dimension 2 property for the singular set of the G-H limit of the sequence $\{  g(t_i) \}$ from the original KRF also holds.

\medskip

{\it Step 3.} backward pseudo-locality, smooth convergence except on co-dimension $2$ set.

In this step we come back to the original KRF.  We will prove the following

\begin{proposition}
\lab{prch-cogi}
 Suppose $(M, g(t_i)) \to_{d_{GH}} (X, d)$, a metric space. Then the following statements hold.

 (i) for any $r>0$ and $x_i \in M$ such that $x_i \to x_\infty \in X$,
 \[
 |B(x_i, r, g(t_i)|_{g(t_i)} \to |B(x_\infty, r)|_X.
 \] Here $|\cdot|_X$ is the n dimensional Hausdorff measure on $X$.

 (ii) for any $x \in X$ and any sequence $\{ r_j \}$ of positive numbers such that $r_j \to 0$, there exists a subsequence of $(Y, r^{-2}_j d, x)$ that converges, in GH sense, to a metric space $(\mathcal{C}_x, d_x, 0 )$ which is a metric cone with vertex $0$.

 (iii) $X= \mathcal{R} \cup \mathcal{S}$ where $\mathcal{S}$ is a closed set of codimension $\ge 2$ and $R$ is consisted of points whose tangent cone is $\mathbb{R}^n$.

 (iv) there exists a $C^\infty$ structure on $\mathcal{R}$ and a $C^\infty$ metric $g_\infty$ on $\mathcal{R}$ which induces the distance $d$ on $X$; moreover, $g(t_i)$ converges to $g_\infty$ in $C^{\infty}_{loc}$ topology on $\mathcal{R}$; the complex structures also converge in $C^{\infty}_{loc}$ topology.
\end{proposition}

 Note that items (i)-(iii) are already proven in Step 2. So we only need to prove (iv).

First let us remind ourself the concept of harmonic radius modeled on \cite{An:1}
\begin{definition} (maximal harmonic radius)
\lab{defharmrad}
Given  numbers $\theta>0$ and  $\a \in (0, 1)$,  and a point $x \in M$ with metric $g$.
 We define
$r^{\theta, \a}_g(x)$ to be the maximum radius $l$ such that there exists,
in the ball $B(x, l)$,  a $C^{1, \a}$ harmonic
coordinate  $\mathbf{x}=(x^1, .., x^n) :
B(x, l) \to {\bf R}^n$,  which satisfies the following properties.

\be \al \lab{th-alcond}
 & a). \quad  e^{-\theta} I \le (g_{pq})
\le e^\theta  I , \, I = (\delta_{pq}); \\
& b). \quad  \sup_{p, q}  l \Vert \partial_i g_{pq} \Vert_{C^0} \le
e^\theta \\
& c). \quad \sup_{p, q} l^{1 + \a}
 \Vert \nb g_{pq} \Vert_{C^{\a}}  \le e^{\theta}.
\eal \ee
\end{definition}

\begin{proposition}  ($C^{1, \a}$  structure of almost Euclidean region).
\lab{prtzqco1.7}
Let $({\M}, g(t))$ be a normalized K\"ahler Ricci flow whose
initial metric $g(0)$ with the basic parameters. There exist
positive numbers $\eta \in (0, 1)$, $\delta$ and $r_0$ such that
the following
 statement holds.

For any $t>0$, $x \in \M$ and $r \in (0, r_0]$, suppose the ball $B(x,
r, g(t))$ is almost Euclidean in volume, i.e.
 \be
 \lab{br>1-e}
 |B(y, r, g(t))|_{g(t)} \ge
(1-\eta) \omega_n r^n,
\ee where $\omega_n$
is the volume of $n$ dimensional Euclidean unit ball.
 Then,
\[ r^{\a, \theta}_{g(t)}(x) \ge
\delta r
\]
\end{proposition}
\proof
By Corollary 1.7 in \cite{TZq2}, the conclusion of the proposition is true under the stronger assumption: for all $B(y, \rho,
g(t)) \subset B(x, r, g(t))$,
\be
|B(y, \rho, g(t))|_{g(t)} \ge
(1-\eta) \omega_n \rho^n.
\ee Now, with the help of the volume comparison result, Corollary \ref{volcomp} (b), we know from  condition \eqref{br>1-e} that for
 all $B(y, \rho,
g(t)) \subset B(x, \delta_0 r, g(t))$,
\be
|B(y, \rho, g(t))|_{g(t)} \ge
(1-2 \eta) \omega_n \rho^n.
\ee Here $\delta_0>0$ is a small but fixed constant. Hence the result is true after adjusting the constants $\eta$ and $\delta$ suitably by Corollary 1.7 in \cite{TZq2}.
\qed

Now we are ready to give a

\proof ( of Proposition \ref{prch-cogi} (iv)).

Pick a point $p \in \mathcal{R} \subset X$. By part (iii) of the proposition, there exists a radius $r=r(p) \le r_0$ such that $B(p, r(p)) \subset \mathcal{R}$ and that
\[
|B(p, r(p))| \ge \omega_n (1-0.5 \eta) r^n(p).
\]Here $r_0$ is given in Proposition \ref{prtzqco1.7}.
Due to the continuity of volume in part (i), for sufficiently large $i$, we have
\[
|B(x_i, r(p), g(t_i))|_{g(t_i)} \ge \omega_n (1-\eta) r^n(p).
\]Here $x_i \in M$ and $x_i \to p$ as $i \to \infty$. By
Proposition
\ref{prtzqco1.7}, we deduce
\[
r^{\a, \theta}_{g(t_i)}(x_i) \ge
\delta r(p).
\]By choosing and fixing the parameter $\theta$ sufficiently small, we can ensure that the isoperimetric constant of $B(x_i, r(p), g(t_i))$ is close enough to the Euclidean one.
Thus we can apply Corollary 1.6 in \cite{BZ:1} to deduce
\[
|Rm|_{g(t_i)} \le C r^{-2}(p)  \quad \text{in} \quad B(x_i, c \delta r(p), g(t_i)) \times [t_i-c \delta^2 r^2(p), t_i+c \delta^2 r^2(p)]
\]for some positive constants $c \le 1, C \ge 1$. The smooth convergence on $\mathcal{R}$ follows after a covering argument. Finally the convergence to a gradient K\"ahler Ricci soliton on $\mathcal{R}$ is standard using Perelman's $W$ entropy, see \cite{Se:1} or \cite{TZz:1} e.g.
\qed

\medskip

{\it Step 4.} co-dimension $4$ property of the singular set $\mathcal{S}$.

We will go back to the conformal metrics $(M, \td g(t_i))=(M, h_i g(t_i))$ and assume that they converge to a metric space $(X, d)$ with the smooth and singular decomposition $Y= \mathcal{\td R}
\cup \mathcal{\td S}$. It suffices to prove $\mathcal{\td S}$ has co-dimension at least $4$ since $C^{1/2}$ norms of $h_i$ are uniformly bounded and $h_i$ is between two positive constants.

First let us recall Theorem 1.23  in \cite{CCT}.

{\it Suppose a sequence of non-collapsed K\"ahler manifolds  $(M, \omega_i)$ converges, G-H, to a metric space $(X, d)$. If the Ricci curvatures, $L^2$ norms of Riemannain curvatures are uniformly bounded, then the codimension of the singular set of $X$ is at least 4.}

 To extend this result to $(M, \td g(t_i))$, there are apparently two obstacles.
 The first is  that $(M, \td g(t_i))$ is not K\"ahler. The second is that we do not know the Ricci curvature is uniformly bounded. However since the conformal factors satisfy $\Vert h_i \Vert_{C^{1/2}} \le C$, the tangent cones of $(X, d)$  are K\"ahler since they are the same as those of the limit space of the K\"ahler manifolds $(M, g(t_i))$. Therefore we only need to prove  the conclusion of Theorem 1.22 in \cite{CCT} which is the corresponding real version of Theorem 1.23. By checking the proof of this theorem, after establishing the $\e$ splitting and metric cone properties and Lemma 2.18, there are only two places where the boundedness of the Ricci curvature is needed. One is the proof of Theorem 3.7 which uses (3.6) that involves the Ricci curvature, the other is in the proof of Theorem 7.2 where the existence of $C^{1, \a}$ harmonic coordinates in the sense of \cite{An:1} is needed. Fortunately, the a priori bounds for $Ric_{\td g}$   in the suitable Morrey  and Kato class in Proposition \ref{pr3.00} is enough to carry out the proof.

 Indeed, the $\e$ splitting map and metric cone property are already done in step 2.2 and 2.4. In the proof of Lemma 2.18 in \cite{CCT}, the boundedness of the Ricci curvature are only  used at two places (2.28) and (2.31). But (2.28) is an immediate consequence of nonhomogeneous mean value inequality whose proof is just like that of Proposition \ref{pr3.1} in the current case. Inequality (2.31) is just the
 Poincar\'e inequality in a ball, which also holds by (5) at the beginning of the section. So the conclusion for Lemma 2.18 is valid for us.

Next, using \eqref{ricwanma}, the method in \cite{TZq2} (Sec. 3, 4) can be applied verbatim to prove the existence of $C^{1, \a}$ harmonic radius. In fact the Morrey type bound is the first lemma in that paper.  Hence the conclusion of Proposition \ref{prtzqco1.7} hold for $(M, \td g)$. Note that proposition is stated for the original metric $(M, g)$. Since $\td g = e^{2 f} g$ and $h$ is not $C^{1, \a}$ we can not immediately claim the result for $(M, \td g)$.

Now let us show that the conclusion of Theorem 3.7 in \cite{CCT} holds for $(M, \td g)$, which states that the regular points of a $\e$ splitting map $\Phi$ on a unit ball $B(x, 1)$  has almost full measure, provided the $Ric \ge -(n-1) \e^2$ and  $B(x, l)$ is $l^{-1}$ close in GH topology to the Euclidean ball for large, fixed $l$. It is a scaled version of the same statement for balls with sufficiently small radius.
During the proof, parts of the Ricci curvature enters via equation (3.6). Due to the sign of the Ricci term in the equation,  only the integral of  the Ricci lower bound $-(n-1) \e^2$ on a 3 ball is used in line 6 p894, which results in smallness of the term.

In our situation, we have neither the lower or upper bound for the Ricci curvature.
However, consider the scaled metric $\hat g = r^{-2} \td g$, $r \in (0, 1]$. We have, by
\eqref{ricwanma}
\be
\lab{richatscale}
\int_{B(x, 3, \hat g)} |Ric_{\hat g}|^2_{\hat g} d \hat g
= r^{-n} \, r^4 \int_{B(x, 3 r, \td g)} |Ric_{\td g}|^2_{\td g} d \td g \le 3^{n-2} A_0 r^2.
\ee Therefore the integral of the full scaled Ricci curvature is small on a 3 ball when the scale is small. This suffices to replace the Ricci lower bound assumption in Theorem 3.7 \cite{CCT}. Thus its conclusion is true for $(M, \td g)$. As mentioned, we can now follows \cite{CCT} verbatim to conclude that the singular set $\mathcal{\td S}$ has co-dimension at least $4$. Note that volume non-collapsing and uniform $L^2$ curvature bound for $(M, \td g)$ are already known by  property 4 in Section 1 and \eqref{rml2wan}.
 This wraps up the proof of the theorem. \qed

\medskip

\section{Remark}

In the first version of the paper, there is an appendix in which we mentioned possible
issues in two previous publications.

One is the lower bound of the Bergman kernel when the metrics blow up in \cite{CW:4}.
Afterwards the authors of \cite{CW:4} posted a note on arXiv disputing our
claim. We have responded to them in more details about our concerns.

The other one is a possible issue in going from Lemma 6.2 to Lemma 6.3 in \cite{Ba:1}.
Recently the author of \cite{Ba:1} informed us that the proof of Lemma 6.3 works and explained his arguments. We agree with it and thank professor R. Bamler for his patient explanation.

 After further consultation and discussion, the appendix in version 1 of the paper is replaced by a supplement with the link  "https://sites.google.com/view/qiszhang-ucr249/home", click " Supplement.arxiv.2509.14820"

\medskip

\hspace{-.5cm}{\bf Acknowledgements}

{\small G. T. is supported by  National Key R.D. Program of China 2020YFA0712800.

Q. S. Z. gratefully acknowledges the support of Simons'
Foundation grant 710364. Thanks also go to Professors R. Bamler, X.X. Chen, B. Wang and Guofang Wei for helpful discussion over the years.

X.H.Z is partially supported by National Key R.D. Program of China 2023YFA1009900 and
2020YFA0712800, and NSFC 12271009.}

\end{document}